\documentclass[12pt]{amsart}
\usepackage{amsmath,amsfonts,amssymb,amscd,mathrsfs,pstricks,pst-plot}

\input xy
\xyoption{all}

\newtheorem{theorem}{Theorem}[section]
\newtheorem{lemma}[theorem]{Lemma}
\newtheorem{corollary}[theorem]{Corollary}
\newtheorem{proposition}[theorem]{Proposition}

\theoremstyle{definition}
\newtheorem{definition}[theorem]{Definition}
\newtheorem{example}[theorem]{Example}
\newtheorem{problem}[theorem]{Problem}
\newtheorem{remark}[theorem]{Remark}

\numberwithin{equation}{section}

\newcommand{\B}{\mathbb{B}}
\newcommand{\C}{\mathbb{C}}
\newcommand{\D}{\mathbb{D}}

\newcommand{\N}{\mathbb{N}}

\newcommand{\Z}{\mathbb{Z}}
\renewcommand{\P}{\mathbb{P}}
\newcommand{\R}{\mathbb{R}}

\renewcommand{\d}{\mathrm{d}}

\newcommand\T{\mathrm{T}}

\newcommand\Id{\mathrm{Id}}

\newcommand{\cA}{\mathcal{A}}

\newcommand{\cC}{\mathcal{C}}

\newcommand{\cL}{\mathcal{L}}

\newcommand{\cO}{\mathcal{O}}

\newcommand{\cU}{\mathcal{U}}

\newcommand{\ggoth}{{\ensuremath{\mathfrak{g}}}}

\newcommand{\Agoth}{{\ensuremath{\mathfrak{A}}}}
\newcommand{\Bgoth}{{\ensuremath{\mathfrak{B}}}}
\newcommand{\Cgoth}{{\ensuremath{\mathfrak{C}}}}

\newcommand{\I}{{\ensuremath{\mathfrak{i}}}}

\newcommand\wt{\widetilde}
\newcommand\hra{\hookrightarrow}

\newcommand\Ker{\mathrm{ker}\,}

\newcommand\di{\partial}

\newcommand\dist{\mathrm{dist}}
\newcommand\dibar{\overline\partial}
\newcommand\bs{\backslash}

\newcommand\spsh{strongly plurisubharmonic}

\newcommand\spsc{strongly pseudoconvex}

\newcommand\wh{\widehat}

\newcommand\rank{\mathrm{rank}}
\newcommand\codim{\mathrm{codim}\,}

\newcommand\reg{\mathrm{reg}}
\newcommand\sing{\mathrm{sing}}

\newcommand{\ol}{\overline}

\newcommand{\lra}{\longrightarrow}
\newcommand\SPCP{strongly pseudoconvex Cartan pair}

\newcommand\CAP{{\rm CAP}}

\newcommand\POP{{\rm POP}}

\begin{document}
\title{Oka manifolds: From Oka to Stein and back}
\author[Franc Forstneri\v c]{Franc Forstneri\v c$^1$}
\footnote{With an appendix by Finnur L\'arusson}

\address{Franc Forstneri\v c, Faculty of Mathematics and Physics, University of Ljubljana, and Institute of Mathematics, Physics and Mechanics, Jadranska 19, 1000 Ljubljana, Slovenia}
\email{franc.forstneric@fmf.uni-lj.si}

\address{Finnur L\'arusson, School of Mathematical Sciences, University of Adelaide, Adelaide SA 5005, Australia}
\email{finnur.larusson@adelaide.edu.au}

\thanks{The author was supported by the grant P1-0291 from ARRS, Republic of Slovenia.  Finnur L\'arusson was supported by Australian Research Council grant DP120104110.}

\date{20 June 2013}


\begin{abstract}
Oka theory has its roots in the classical Oka-Grauert principle whose main result is Grauert's  classification of principal holomorphic fiber bundles over Stein spaces.  Modern Oka theory concerns holomorphic maps from Stein manifolds and Stein spaces to Oka manifolds. It has emerged as a subfield of complex geometry in its own right since the appearance of a seminal paper of M.\ Gromov in 1989. 

In this expository paper we discuss Oka manifolds and Oka maps.  We describe equivalent characterizations of Oka manifolds, the functorial properties of this class, and geometric sufficient conditions for being Oka, the most important of which is Gromov's ellipticity. We survey the current status of the theory in terms of known examples of Oka manifolds, mention open problems and outline the proofs of the main results. In the appendix by F.\ L\'arusson it is explained how Oka manifolds and Oka maps, along with Stein manifolds, fit into an abstract homotopy-theoretic framework.  

The article is an expanded version of lectures given by the author at Winter School KAWA 4 in Toulouse, France, in January 2013. A comprehensive exposition of Oka theory is available in the monograph \cite{FF:book}.
\end{abstract}

\maketitle
\tableofcontents

%%%%%%%%%%%%%%%%%%%%%%%%%%%%%%%%%%%%%%%%%%%%%%%%%
%																								%
%   Introduction            										%
%																								%
%%%%%%%%%%%%%%%%%%%%%%%%%%%%%%%%%%%%%%%%%%%%%%%%%

\section{Introduction}
\label{sec:Intro}
Oka theory is about a tight relationship between homotopy theory and complex geometry involving {\em Stein manifolds} and {\em Oka manifolds}. It has a long and rich history, beginning with Kiyoshi Oka in 1939, continued by Hans Grauert and the German  school in the late 1950's and 1960's, revitalized by Mikhael Gromov in 1989, and leading to an introduction and systematic study of {\em Oka manifolds} and {\em Oka maps} in the last decade.

The heuristic {\em Oka principle} says that there are only topological obstructions to solving complex-analytic problems on Stein spaces that can be cohomologically, or even homotopically, formulated. A classical example is the {\em Oka-Grauert principle} (Grauert \cite{Grauert3}; see also Cartan \cite{Cartan1958} and Henkin and Leiterer \cite{Henkin-Leiterer:Oka}): For any complex Lie group $G$, the holomorphic classification of principal $G$-bundles over any Stein space agrees with their topological classification. The same holds for fiber bundles with $G$-homogeneous fibers; in particular, for complex vector bundles (take $G=GL_k(\C)$). The special case of line bundles $(k=1,\ G=\C^*=\C\setminus\{0\})$ is a theorem of Oka \cite{Oka1939} from 1939 which marks the beginning of Oka theory. 

Since a fiber bundle is defined by a 1-cocyle of transition maps, it is not surprising that the original formulation of the Oka-Grauert principle is cohomological. However, it was already observed by Henri Cartan \cite{Cartan1958} in 1958 (the year of publication of Grauert's main paper \cite{Grauert3} on this subject) that the result can be phrased in terms of the existence of holomorphic sections $X\to Z$ of certain associated fiber bundles $\pi\colon Z\to X$ with Lie group fibers over a Stein base $X$. More precisely, the key problem is to find a holomorphic section homotopic to a given continuous section. 

It is this homotopy-theoretic point of view that was adopted and successfully exploited by Mikhail Gromov in his seminal paper \cite{Gromov:OP} in 1989. (A complete exposition of his work first appeared in \cite{Forstneric-Prezelj2000, Forstneric-Prezelj2001, Forstneric-Prezelj2002}.) This change of philosophy, together with the introduction of substantially weaker sufficient conditions, liberated the Oka principle from the realm of fiber bundles with homogeneous fibers, thereby making it much more flexible and substantially more useful in applications. In particular, a proof of the embedding theorem for Stein manifolds into Euclidean spaces of minimal dimension, due to Eliashberg and Gromov \cite{Eliashberg-Gromov1, Eliashberg-Gromov2} and Sch\"urmann \cite{Schurmann}, became viable only in the wake of Gromov's Oka principle. For this and other applications see \cite[Chap.\ 8]{FF:book}.  

The modern Oka principle focuses on those analytic properties of a complex manifold $Y$ which ensure that every continuous map $X\to Y$ from a Stein space $X$ is homotopic to a holomorphic map, with certain natural additions (approximation, interpolation, the inclusion of a parameter) that are motivated by classical function theory on Stein spaces. Specifically, we say that a complex manifold $Y$ enjoys the {\em weak homotopy equivalence principle} if for every Stein space $X$, the inclusion $\iota \colon \cO(X,Y) \hra\cC(X,Y)$ of the space of all holomorphic maps $X\to Y$ into the space of all continuous maps is a {\em weak homotopy equivalence} with respect to the compact-open topology, that is, $\iota$ induces isomorphisms of all homotopy groups:
\begin{equation}
\label{eq:whe}
	\pi_k(\iota) \colon \pi_k( \cO(X,Y)) \stackrel{\cong}{\lra} \pi_k(\cC(X,Y)),\quad k=0,1,2,\ldots.
\end{equation}

The analogous questions are considered for sections of holomorphic submersions $\pi\colon Z\to X$ onto Stein spaces $X$. Gromov's main result in \cite{Gromov:OP} is that the existence of a {\em holomorphic fiber-dominating spray} on $Z|_U:=\pi^{-1}(U)$ over small open subsets $U$ of a Stein base space $X$ implies all forms of the Oka principle for sections $X\to Z$ (Theorem \ref{SES:OP} in \S \ref{ss:ellipticsub} below). Submersions with this property are said to be {\em elliptic}. In particular, a complex manifold $Y$ with a dominating holomorphic spray -- an {\em elliptic manifold} -- enjoys all forms of the Oka principle for maps $X\to Y$ from Stein spaces.
Although ellipticity is a useful geometric sufficient condition for validity of the Oka principle, it is still not clear whether it is also necessary, and not many interesting functorial properties have been discovered for the class of elliptic manifolds. 

In the papers \cite{FF:CAP,FF:EOP,FF:OkaManifolds} it was proved that a complex manifold $Y$ satisfies the Oka principle (with approximation, interpolation, parametric)  if (and only if) it has the so-called {\em convex approximation property} (CAP) that was first introduced in \cite{FF:CAP}: 

CAP of $Y$: {\em Any holomorphic map $K\to Y$ from a compact convex set $K$ in $\C^n$ to $Y$ is a uniform limit of entire maps $\C^n\to Y$}. 

(See Def.\ \ref{def:CAP} and Theorem \ref{th:CAP} below.) This in particular answered a question raised by Gromov \cite[3.4.(D), p.\ 881]{Gromov:OP}. A complex manifold enjoying these equivalent properties is said to be an {\em Oka manifold}, a term that was first introduced in \cite{FF:OkaManifolds} (see also \cite{Larusson2}). Every complex homogeneous manifold is an Oka manifold, so Theorem \ref{th:CAP} includes the classical Oka-Grauert theory. In the wake of this simple characterization it became clear, mainly through the work of F.\ L\'arusson and the author, that the class of Oka manifolds satisfies several nontrivial functorial properties; see \S \ref{ss:examples}. The picture was completed in \cite{FF:OkaMaps} where the analogous results were obtained for the class of {\em Oka maps}; see \S \ref{ss:Okamaps}. 

F.\ L\'arusson constructed an underlying model structure that shows that the Oka property is, in a precise sense, homotopy-theoretic; see \S\ref{sec:Appendix}. 

Although this expository article is largely based on the recent monograph \cite{FF:book}, we use this occasion to survey the developments since its publication in August 2011. Here is a sampling of new results: examples due to A.\ Hanysz of Oka hyperplane complements in projective spaces (\S \ref{ss:hypersurfaces}); F.\ L\'arusson's proof that smooth toric varieties are Oka (\S \ref{ss:toric}); the notion of a stratified Oka manifold and related results on Kummer surfaces (\S \ref{ss:stratOka}); a survey of the Oka property for compact complex surfaces, with emphasis on class VII (\S \ref{ss:Class7}). There is a growing list of applications of the modern Oka principle for which the older theory does not suffice; see Chapters 7 and 8 of \cite{FF:book} and the preprint \cite{Alarcon-Forstneric} on {\em directed immersions} of open Riemann surfaces into Euclidean spaces.

%%%%%%%%%%%%%%%%%%%%%%%%%%%%%%%%%%%%%%%%%%%%%%%%%
%																								%
%   SECTION: OKA MANIFOLDS	 										%
%																								%
%%%%%%%%%%%%%%%%%%%%%%%%%%%%%%%%%%%%%%%%%%%%%%%%%

\section{Oka manifolds, Oka maps, and elliptic submersions}
\label{sec:OM}
The main theme of modern Oka theory is the study of those complex analytic properties of a complex manifold $Y$ which say that there exist `many' holomorphic maps $X \to Y$ from any Stein manifold (or Stein space) $X$. These {\em Oka properties} indicate that $Y$ is in a certain sense holomorphically large, or, as we prefer to say, {\em holomorphically flexible}. Oka properties are modeled upon the classical results concerning holomorphic functions on Stein manifolds, considered as maps $X\to Y=\C$ to the complex number field. They are opposite to holomorphic rigidity properties, the latter being commonly expressed by various types of hyperbolicity. 

We begin in \S \ref{ss:flex} with a survey of the main hyperbolicity and flexibility properties. In \S \ref{ss:Oka} we introduce our basic flexibility property, the {\em Convex Approximation Property} (CAP), which is essentially dominability with approximation on compact geometrically convex sets in $\C^n$. We then state our main result, Theorem \ref{th:CAP}, to the effect that CAP implies, and hence is equivalent to, all Oka properties that have been considered in the literature. A complex manifold satisfying these equivalent properties is said to be an {\em Oka manifold}. The proof of Theorem \ref{th:CAP} is outlined in \S \ref{sec:methods}. In \S \ref{ss:examples}--\S \ref{ss:Class7} we give examples and functorial properties of the class of Oka manifolds. In \S \ref{ss:ellipticsub} we describe Gromov's Oka principle for sections of elliptic submersions over Stein spaces. A proof of this result and generalizations can be found in \cite[Chap.\ 6]{FF:book} and in the papers \cite{Forstneric-Prezelj2001, Forstneric-Prezelj2002,FF:Subelliptic,FF:Kohn}.

Sections \ref{sec:OM} and \ref{sec:methods} together constitute an abridged version of Chapters 5 and 6 in author's monograph \cite{FF:book}, but with the addition of some new results.

%%%%%%%%%%%%%%%%%%%%%%%%%%%%%%%%%%%%%%%%%%%%%%%%%%%%%%%%%%%
%																													%
%   COMPLEX MANIFOLDS: FLEXIBILITY VERSUS RIGIDITY				%
%																													%
%%%%%%%%%%%%%%%%%%%%%%%%%%%%%%%%%%%%%%%%%%%%%%%%%%%%%%%%%%%
\subsection{Complex manifolds: flexibility versus rigidity}
\label{ss:flex}
A main feature distinguishing complex geometry from smooth geometry is the phenomenon of {\em holomorphic rigidity}. 

In smooth geometry there is no rigidity, unless we introduce some additional structure (a Riemannian metric, a symplectic structure, a contact structure, etc.). Indeed, every continuous map between smooth manifolds can be approximated in all natural topologies by smooth maps; the analogous statement also holds in the real-analytic category. In particular, every homotopy class of maps is represented by a smooth map. 

This is not at all the case in the holomorphic category. For example, Liouville's theorem tells us that there are no nonconstant holomorphic maps of $\C$ to the disc $\D=\{z\in \C\colon |z|<1\}$, or to any bounded domain in $\C$. Furthermore, Picard's theorem says that there are no nonconstant holomorphic maps $\C\to\C\setminus\{0,1\}$. Looking at annuli $A_r=\{z\in \C\colon 1/r <|z|< r\}$ for $1< r \le\infty$, there is a holomorphic map $A_r\to A_R$ of degree $k\in\Z$ (in the homotopy class of the map $z\mapsto z^k/|z|^k$) if and only if $r^{|k|}\le R$; in this case the map $z\mapsto z^k$ is such. Looking at maps to higher dimensional manifolds, {\em Green's theorem} \cite{Green} says that every holomorphic map $\C\to \C\P^n$ whose range omits $2n+1$ hyperplanes in general position is a constant map.

The quantitative version of holomorphic rigidity of the disc $\D$ is expressed by the classical {\em Schwarz-Pick lemma}: Every holomorphic map $f\colon \D\to\D$ satisfies 
\[
	\frac{|\d f(z)|}{1-|f(z)|^2} \le \frac{|\d z|}{1-|z|^2},\quad z\in\D. 
\]
In particular, $|f'(0)|\le 1-|f(0)|^2$. This inequality says that holomorphic self-mappings of the disc $\D$ are distance decreasing in the Poincar\'e metric on $\D$; orientation preserving isometries are precisely the holomorphic automorphisms of $\D$.
 
These and other classical holomorphic rigidity theorems lead to the notion of a hyperbolic manifold.
A complex manifold $Y$ is said to be 

\begin{itemize}
\item[---]
{\em Brody hyperbolic} \cite{Brody} if every holomorphic map $\C\to Y$ is constant;

\item[---]
{\em Brody volume hyperbolic} if every holomorphic map $\C^n\to Y$ $(n=\dim Y)$ has rank $<n$ at each point (a degenerate map);

\item[---]
{\em  Kobayashi (complete) hyperbolic}  \cite{Kobayashi1,Kobayashi2} if the {\em Kobayashi pseudometric} $k_Y$ is a (complete) metric on $Y$. 
\end{itemize}

Recall that $k_Y$ is the integrated form of the infinitesimal pseudometric
\[
	|v| = \inf\left\{ 
	\frac{1}{|\lambda|} \colon \ f\colon \D\to Y\ \hbox{holomorphic},\ f(0)=y,\ f'(0)=\lambda v \right\},
	\quad v\in T_y Y.
\]
A Kobayashi hyperbolic manifold is also Brody hyperbolic;  the converse holds for compact manifolds (Brody \cite{Brody}). Furthermore, for every integer $k\in \{1,\ldots,\dim Y\}$ there is a quantitative notion of {\em Eisenman $k$-hyperbolicity} (Eisenman \cite{Eisenman}); for $k=1$ it coincides with Kobayashi hyperbolicity.  

Most complex manifolds have at least some amount of rigidity. In particular, every compact complex manifold $Y$ of general type (i.e., of Kodaira dimension $\kappa_Y=\dim Y$) is volume hyperbolic (Kobayashi and Ochiai \cite[Theorem~2]{Kobayashi-Ochiai}), and it is believed to be almost Kobayashi hyperbolic, in the sense that it admits a proper complex subvariety $Y'$ which contains the image of any nonconstant entire map $\C\to Y$. 

Contrary to hyperbolicity, holomorphic flexibility properties should signify the existence of many holomorphic maps $\C\to Y$ and, more generally, $\C^n \to Y$ for any $n\ge 1$. 

The key is to specify the precise meaning of the word {\em many} in this context. 

For example, taking $n=1$, one can ask that for every point $y\in Y$ and tangent vector $v\in T_y Y$ there exist an entire map $f\colon \C\to Y$ with $f(0)=y$ and $f'(y)=v$. Not much seems known about this particular notion of anti-hyperbolicity. One can also demand that any pair of points in $Y$ can be connected by a complex line $\C\to Y$; such a manifold is said to be {\em strongly $\C$-connected}. A weaker version of $\C$-connectedness is that any pair of points is connected by a finite chain of entire lines, in analogy with {\em rational connectedness} where $\C$ is replaced by $\C\P^1$.  

A stronger flexibility property, which is opposite to Brody volume hyperbolicity, is {\em dominability}. A connected complex manifold $Y$ is said to be {\em dominable} if there exists a holomorphic map $f\colon \C^n\to Y$ ($n=\dim Y$) such that $\rank_z f=n$ at some point $z\in \C^n$ (and hence at most points). If $f$ can be chosen so that $f(0)=y\in Y$ and $\rank_0 f=n$, then $Y$ is said to be {\em dominable at $y$}; it is {\em strongly dominable} if it is dominable at each point. The theorem of Kobayashi and Ochiai \cite{Kobayashi-Ochiai} then says that a compact complex manifold $Y$ of general type is not dominable.

If we strengthen dominability even further by demanding the existence of a holomorphically varying family of dominating maps $f_y\colon \C^n\to Y$ ($y\in Y$) with $f_y(0)=y$, we arrive at Gromov's notion of a {\em dominating holomorphic spray} on $Y$ (see \S \ref{ss:elliptic} below). A complex manifold which admits a dominating spray is said to be {\em elliptic}. In particular, every complex Lie group and every complex homogeneous manifold is elliptic; a dominating spray is provided by the exponential map.

Another source of motivation to study flexibility properties comes from the classical theory of holomorphic functions on Stein spaces $X$, which we treat as holomorphic mappings $X\to\C$. The {\em Oka-Weil theorem} extends the classical {\em Runge theorem}: given a compact holomorphically convex set $K$ in a Stein space $X$, every holomorphic function on a neighborhood of $K$ can be approximated, uniformly on $K$, by entire functions $X\to \C$.  Similarly, given a closed complex (hence Stein) subvariety $X'\subset X$, every holomorphic function $X'\to \C$ extends to a holomorphic function $X\to \C$; this is the {\em Cartan extension theorem}. These two results may be combined: given a function $f\colon  K\cup X' \to\C$ that is holomorphic on a subvariety $X'$ and on a neighborhood of a compact holomorphically convex set $K$ in a Stein space $X$, there is a holomorphic function $F\colon X\to \C$ which agrees with $f$ on $X'$ and approximates $f$ uniformly on $K$. 

It is now the time to turn the tables --- replacing the target $\C$ by an arbitrary complex manifold $Y$, these classical theorems of complex analysis define various {\em Oka properties} which $Y$ may or may not have. However, due to possibly nontrivial topology of both the source and the target manifold, the question must now be asked differently: whether the extension or the approximation of holomorphic maps $X\to Y$ is possible in the absence of topological obstructions. 

The most basic Oka property of a complex manifold $Y$ is that every continuous map $X\to Y$ from a Stein space $X$ is homotopic to a holomorphic map. This always holds when $Y$ is contractible; for example, for the unit disc. Winkelmann \cite{Winkelmann} completely classified this propery for maps $X\to Y$ between Riemann surfaces, even when the source $X$ is compact (i.e., non-Stein). Recently this property was studied under the name {\em h-principle}, abbreviated $hP(Y)$, by Campana and Winkelmann \cite{Campana-Win}.  Their main result, Theorem 6.1, says that a complex projective manifold which satisfies the h-principle is {\em special} in the sense of Campana \cite{Campana2004,Campana2004-2}. Being special is an anti-hyperbolicity or anti-general type notion defined for compact K\"ahler manifolds. Campana has proved that $Y$ is special if $Y$ is either dominable, rationally connected, or has Kodaira dimension $0$; on the other hand, no compact manifold of general type is special.  Campana has conjectured that a complex manifold is special if and only if it is $\C$-connected if and only if its Kobayashi pseudo-metric vanishes identically.

We now consider stronger Oka properties; see \cite[\S 5.15]{FF:book} for precise definitions. 

By adding the approximation condition, we ask that any continuous map $f:X\to Y$ from a Stein space $X$ that is holomorphic on a neighborhood of a compact $\cO(X)$-convex set $K\subset X$ be homotopic to a holomorphic map $X\to Y$ by a homotopy that is holomorphic near $K$ and uniformly as close as desired to $f$ on $K$. A manifold $Y$ which satisfies this property for any triple $(X,K,f)$ is said to enjoy the {\em basic Oka property with approximation}, BOPA. Similarly one defines {\em the basic Oka property with interpolation}, BOPI, and the combination of the two, BOPAI. A restricted version of BOPI is CIP, the {\em convex interpolation property}, introduced by F.\ L\'arusson \cite{Larusson5}. One can strengthen the interpolation condition along a subvariety to include jet interpolation; this gives properties BOPJI and BOPAJI. 

The corresponding properties for families of maps, depending continuously on a parameter, are called {\em parametric Oka properties}; so we speak of POPA, POPI, POPAI, etc. More precisely, assume that $P$ is a compact Hausdorff space (the parameter space) and $P_0$ is a closed subset of $P$. A complex manifold $Y$ enjoys the {\em parametric Oka property} (POP) if for every Stein space $X$, any continuous map $f\colon P\to \cC(X,Y)$ such that $f(P_0)\subset \cO(X,Y)$ can be deformed to a continuous map $f_1 \colon P\to \cO(X,Y)$ by a homotopy that is fixed on $P_0$:
% (cf.\ \cite[p.\ 194]{FF:book}):
\vskip -5mm
\[
	\xymatrix{ 
	  P_0  \ar[d]_{incl}  \ar[r] & \cO(X,Y) \ar[d]^{incl} \\
	  P    \ar[r]^{\!\!\!\!\!\! f}  \ar[ur]^{f_1}	 & \cC(X,Y) } 
\]
In some theorems (in particular, in Theorem \ref{th:CAP}) we restrict POP to parameter spaces which are Euclidean compacta; this restriction is inessential in most applications. By purely homotopy-theoretic arguments, the validity of POP for all pairs $P_0\subset P$ of Euclidean compacta (for a given manifold $Y$) implies POP also when $P_0$ is a subcomplex of a possibly infinite CW complex $P$ \cite[\S 16]{Larusson2}. Approximation (on a compact holomorphically convex set) and interpolation (on a closed complex subvariety) can naturally be included in this picture. The basic (non-para\-metric) case corresponds to $P=\{p\}$ a singleton and $P_0=\varnothing$. The {\em $1$-parametric Oka property} refers to the case $P=[0,1]$ and $P_0=\{0,1\}$; the conclusion is that any pair of holomorphic maps $f_0,f_1\colon X\to Y$ which are homotopic through continuous maps are also homotopic through holomorphic maps.

This is all very well, but are there any nontrivial examples of complex manifolds % (besides of course the complex Euclidean spaces)
that satisfy these Oka properties? 

Indeed, Grauert proved back in 1958 \cite{Grauert1,Grauert2,Grauert3} that every complex Lie group and, more generally, every complex homogeneous manifold, enjoys all of the properties described above. (See Theorem \ref{CAP:Lie} below.) Grauert's results can be considered as the true beginning of Oka theory. Thirty years later, Gromov \cite{Gromov:OP} introduced the class of elliptic manifolds and proved that they also satisfy all Oka properties (see \S\ref{ss:elliptic} below). 

The continuation is an ongoing story. Reader, follow me to the main results of modern Oka theory.

%%%%%%%%%%%%%%%%%%%%%%%%%%%%%%%%%%%%%%%%%%%%%%%%%%%%%%%%%%%%%%%%%%%
%																																	%
%  CAP AND OKA MANIFOLDS																					%
%																																	%
%%%%%%%%%%%%%%%%%%%%%%%%%%%%%%%%%%%%%%%%%%%%%%%%%%%%%%%%%%%%%%%%%%%

\subsection{The convex approximation property and Oka manifolds}
\label{ss:Oka}
The simplest of several equivalent characterizations of Oka manifolds is the following Runge approximation property for maps from Euclidean spaces which was introduced in \cite{FF:CAP}. 

\begin{definition}
\label{def:CAP}  
A complex manifold $Y$ is said to enjoy the {\em convex approximation property}, $\CAP$, if every holomorphic map $f\colon K\to Y$ from (a neighborhood of) a compact convex set $K\subset \C^n$ can be approximated, uniformly on $K$, by entire maps $\C^n\to Y$.
\end{definition}

By fixing the integer $n\in\N$ in the above definition we get the $n$-dimensional {\em convex approximation property}, $\CAP_n$. Clearly $\CAP_{n+1}\Rightarrow \CAP_{n}$ for every $n=1,2,\ldots$; a manifold satisfies $\CAP$ if and only if it satisfies $\CAP_n$ for every $n$. (In \cite{FF:CAP} the testing class of convex sets $K\subset \C^n$ was restricted to a certain subclass consisting of {\em special convex sets}, mainly to facilitate the verification of $\CAP$ in concrete examples. We shall not bother here with this minor technical distinction.) We  notice that, to verify $\CAP$, it suffices to prove that for any pair of compact convex sets $K\subset Q \subset \C^n$, holomorphic maps $K\to Y$ can be approximated by holomorphic maps $Q\to Y$.

The following can be considered as one of the main results of modern Oka theory. For a more precise statement see \cite[Theorem 5.4.4, p.\ 193]{FF:book}. The parametric Oka property (POP) now refers to parameter spaces that are Euclidean compacta.

%
%
%  CAP IMPLIES OKA
%
%
\begin{theorem}[\bf The main theorem]
\label{th:CAP}
If a complex manifold $Y$ enjoys $\CAP$, then maps $X\to Y$ from any reduced Stein space $X$ satisfy all forms of the Oka principle (with approximation, (jet-) interpolation, parametric). The same is true for sections $X\to Z$ of any stratified holomorphic fiber bundle $\pi\colon Z\to X$ over a Stein base $X$ all of whose fibers over different strata enjoy $\CAP$. 
\end{theorem}

Recall that a holomorphic map $\pi\colon Z\to X$ of complex spaces is said to be a {\em holomorphic submersion} if locally near each point $z\in Z$ it is fiberwise isomorphic to a projection map $U\times V\to U$, where $U\subset X$ is an open neighborhood of the projected point $x=\pi(z)\in X$ and $V$ is smooth (nonsingular). It follows that the fibers $Z_x=\pi^{-1}(x)$ $(x\in X)$ of a holomorphic submersion $\pi\colon Z\to X$ are complex manifolds. 

A surjective holomorphic submersion $\pi\colon Z\to X$ is said to be a {\em stratified holomorphic fiber bundle} if $X$ admits a stratification $X=X_0\supset X_1\supset \cdots \supset X_m =\varnothing$ by closed complex subvarieties such that every difference $S_j=X_j\setminus X_{j+1}$ is smooth (nonsingular), and the restriction $\pi\colon Z|_S\to S$ to any connected component $S$ (stratum) of any $S_j$ is a holomorphic fiber bundle. The fibers over different strata may be different.

Since $\CAP$ is clearly a special case of the basic Oka property with approximation, applied with the pair $X=\C^n$ and $K$ a compact convex set in $\C^n$, $\CAP$ is also necessary in Theorem \ref{th:CAP}. Furthermore, it turns out that all individual Oka properties introduced in the previous section (such as BOPA, BOPI, BOPAI, and their parametric analogues) are pairwise equivalent; see \cite{FF:EOP} and especially \cite[\S 5.15]{FF:book}.

\begin{definition}[\bf Oka manifolds]
\label{def:Oka}
A complex manifold $Y$ is said to be an {\em Oka manifold} if it satisfies any (and hence all) of the pairwise equivalent Oka properties in Theorem \ref{th:CAP}. In particular, a complex manifold is Oka if and only if it enjoys $\CAP$.
\end{definition}

The following is a corollary to Theorem \ref{th:CAP} (cf.\ \cite[Corollary 5.4.8]{FF:book}).

\begin{corollary}[\bf The weak homotopy equivalence principle]
\label{cor:weak}
If $X$ is a reduced Stein space and $\pi \colon Z\to X$ is a stratified holomorphic fiber bundle whose fibers are Oka manifolds, then the inclusion 
\begin{equation}
\label{inclusion}
    \iota\colon \Gamma_{\cO}(X,Z) \hra \Gamma_{\cC}(X,Z)
\end{equation}
of the space of holomorphic sections of $\pi$ into the space of continuous sections is a {\em weak homotopy equivalence}, i.e., the induced map $\pi_k(\iota)\colon \pi_k(\Gamma_{\cO}(X,Z)) \to  \pi_k(\Gamma_{\cC}(X,Z))$ of homotopy groups is an isomorphism for every $k=0,1,2,\ldots$. 
\end{corollary}

\begin{proof}
Denote by $D^{k+1}$ the closed unit ball in $\R^{k+1}$ and by $S^k=bD^{k+1}$ the $k$-dimensional sphere. Applying the parametric Oka property (Theorem \ref{th:CAP}) with the parameter spaces $P_0=\varnothing \subset S^k=P$ we see that every continuous map $S^k\to \Gamma_{\cC}(X,Z)$ can be deformed to a continuous map $S^k\to \Gamma_{\cO}(X,Z)$; this says that the homomorphism $\pi_k(\iota)$ is surjective. On the other hand, applying Theorem \ref{th:CAP} with the parameter spaces $P_0=S^k \subset D^{k+1}=P$ we conclude that every map $f\colon D^{k+1} \to \Gamma_{\cC}(X,Z)$ satisfying $f(S^k)\subset \Gamma_{\cO}(X,Z)$ can be deformed to a map $D^{k+1} \to \Gamma_{\cO}(X,Z)$ by a homotopy that is fixed on $S^k$; this means that $\pi_k(\iota)$ is injective. 
\end{proof}

Theorem \ref{th:CAP} implies a homotopy principle for liftings of holomorphic maps. Given a map $\pi\colon E\to B$, we say that a map $F \colon X\to E$ is a {\em lifting} of a map $f \colon X\to  B$ if $\pi\circ F=f$. Similarly one defines homotopies of liftings.

%
%
%   OP - LIFTING
%
%
\begin{corollary}[\bf The Oka principle for liftings]
\label{OP:lifting}
{\rm \cite[Corollary 5.4.11]{FF:book}}
Assume that $\pi\colon E\to B$ is a stratified holomorphic fiber bundle all of whose fibers are Oka manifolds. If $X$ is a Stein space and $f\colon X\to B$ is a holomorphic map, then any continuous lifting $F_0\colon X\to Z$ of $f$ admits a homotopy of liftings $F_t\colon X\to E$ $(t\in [0,1])$ such that $F_1$ is holomorphic. The corresponding result holds with approximation and interpolation.
\end{corollary}

\begin{proof}
Assume that $\pi \colon E\to B$ is a holomorphic fiber bundle with Oka fiber $Y$. Let $\pi' \colon f^*E\to X$ denote the pull-back bundle whose fiber over a point $x\in X$ is $E_{f(x)} \cong Y$.
\[
    \xymatrix{ f^* E \ar[d]^{\pi'} \ar[r] & E \ar[d]^{\pi} & \ar[l] Y \\
               X \ar[r]^{\ \ f} \ar[ur]^{F_t} & B }
\]
Sections $X\to f^*E$ are in bijective correspondence with liftings $F\colon X\to E$ of $f$. Since the fiber $Y$ of $\pi' \colon f^*E\to X$ is Oka, the conclusion follows from Theorem \ref{th:CAP}. In the general case we stratify $X$ so that the strata are mapped by $f$ to the strata of $B$; then $f^* E$  is also a stratified fiber bundle
over $X$ and we conclude the proof as before. 
\end{proof}

%%%%%%%%%%%%%%%%%%%%%%%%%%%%%%%%%%%%%%%%%%%%%%%%%%%%%%%%%%%%%%%%%%%%%%%%%
%																																				%
%  EXAMPLES AND FUNCTORIAL PROPERTIES																		%
%																																				%
%%%%%%%%%%%%%%%%%%%%%%%%%%%%%%%%%%%%%%%%%%%%%%%%%%%%%%%%%%%%%%%%%%%%%%%%%	

\subsection{Examples and functorial properties of Oka manifolds}
\label{ss:examples}
The examples in this subsection are mainly taken from \cite[\S 5.5]{FF:book} and \cite{FLflex, Larusson-toric}. Those sources contain a wealth of additional information and references to the original papers. In particular, the results in \cite{FLflex} answer some of the questions left open in \cite{FF:book}.

%%%%%%%%%%%%%%%%%%%%%%%%%%%%%%%%%%%%%%%%%%%%%%%%%%%%%%%%%%%%%%%
%																															%	
%   HOMOGENEOUS MANIFOLDS            													%
%																															%
%%%%%%%%%%%%%%%%%%%%%%%%%%%%%%%%%%%%%%%%%%%%%%%%%%%%%%%%%%%%%%%
	
\subsubsection{Complex homogeneous manifolds}
\label{OP;examples;homogeneous}

\begin{theorem}
\label{CAP:Lie}
{\rm (Grauert \cite{Grauert1})}
Every complex Lie group and, more generally, every complex homogeneous manifold is an Oka manifold.  
\end{theorem}

\begin{proof}
Let $G$ be a complex Lie group with the identity element $1\in G$. In view of Theorem \ref{th:CAP} it suffices to show that $G$ enjoys CAP. Denote by $\ggoth =T_1 G$ the Lie algebra of $G$ and by $\exp \colon \ggoth\to G$ the exponential map. Assume that $K$ is a compact convex set in $\C^n$ and $f\colon K\to G$ is a holomorphic map in an open neighborhood of $K$. 

If $f(K) \subset G$ lies sufficiently close to $1\in G$, then $f=\exp(h)$ for a unique holomorphic map $h \colon K \to {\ggoth}$. Approximating $h$ uniformly on $K$ by an entire map $\wt h\colon\C^n\to {\ggoth}$ and taking $\wt f=\exp \wt h \colon \C^n \to G$ yields an entire map approximating $f$. 

In general we split $f$ into a finite product of maps close to the identity to which the previous argument applies. We may assume that the origin $0 \in \C^n$ is contained in the interior of $K$. Set $f_t(z)=f(tz)$ for $t\in [0,1]$; then $f_1=f$ and $f_0$ is the constant map $\C^n \ni z\to f(0)\in G$. Choose a large integer $N\in\N$ and write
\[
    f= f_1= f_1 \bigl(f_{\frac{N-1}{N}} \bigr)^{-1}
    \cdotp f_{\frac{N-1}{N}} \bigl(f_{\frac{N-2}{N}} \bigr)^{-1}
    \cdots f_{\frac{1}{N}} (f_0)^{-1}\cdotp f_0.
\]
If $N$ is large enough, then each of the quotients $f_{\frac{k}{N}} \bigl(f_{\frac{k-1}{N}} \bigr)^{-1}$ is sufficiently close to $1$ so that it admits a holomorphic logarithm $h_k\colon K \to \ggoth$. Approximating $h_k$ by an entire map $\wt h_k\colon\C^n\to \ggoth$ and taking $\wt g_k=\exp \wt h_k$ and
$\wt f = \wt g_N  \wt g_{N-1} \cdots \wt g_{1}  f_0 \colon \C^n\to G$ gives an entire map approximating $f$.

The proof for a complex homogeneous manifold is similar; see \cite[p.\ 198]{FF:book}.
\end{proof}

\begin{example}
Every complex projective space $\C\P^n$ and, more generally, every complex Grassmannian, is an Oka manifold.
\end{example}

%%%%%%%%%%%%%%%%%%%%%%%%%%%%%%%%%%%%%%%%%%%%%%%%%%%
%																									%
%  ASCENT AND DESCENT															%		
%																									%
%%%%%%%%%%%%%%%%%%%%%%%%%%%%%%%%%%%%%%%%%%%%%%%%%%%		
\subsubsection{Holomorphic fiber bundles with Oka fibers}
\label{OP;examples;ascent-descent}
The class of Oka manifolds is stable under unramified holomorphic coverings and quotients. 

%
%
%   COVERING SPACES
%
%
\begin{proposition}
\label{OP-covering}
{\rm \cite[Proposition 5.5.2, p.\ 199]{FF:book}}
If $\wh Y\to Y$ is a holomorphic covering map of complex manifolds, then $Y$ is an Oka manifold if and only if
$\wh Y$ is an Oka manifold. In particular, any unramified holomorphic quotient of $\C^n$ is Oka; this includes all tori $\C^n/\Gamma$, where $\Gamma$ is a lattice in $\C^n$.
\end{proposition}

The proof is immediate since any continuous map $f\colon K\to Y$ from a convex set $K\subset \C^n$ lifts to a map $K\to \wh Y$, and a holomorphic map lifts to a holomorphic map. This shows that $Y$ and $\wh Y$  enjoy CAP at the same time. (See \cite[p.\ 199]{FF:book}.)

%Proposition \ref{OP-covering}, together with the uniformization theorem for Riemann surfaces, implies the following. 

A complex manifold is called \textit{Liouville} if it carries no nonconstant negative plurisubharmonic functions, and \textit{strongly Liouville} if its universal covering space is Liouville. From Proposition \ref{OP-covering} and the fact that $\C^n$ is Liouville we easily infer the following.

\begin{corollary}
\label{cor:Liouville}
Every Oka manifold is strongly Liouville.  
\end{corollary}

\begin{corollary} 
\label{p6.1}    
{\rm \cite[Corollary 5.5.3, p.\ 199]{FF:book}}                                     
A Riemann surface $Y$ is an Oka manifold if and only if it is not Kobayashi hyperbolic, and this holds if and only if $Y$ is one of the Riemann surfaces $\C\P^1$, $\C$, $\C^*$, or a torus $\C/\Gamma$.
\end{corollary}

\begin{proof} 
By the Riemann-Koebe uniformization theorem, the universal covering space of any Riemann surface is one of the Riemann surfaces $\C\P^1$, $\C$, or the disc $\D$. The Riemann sphere $\C\P^1$ is homogeneous and hence Oka; it has no nontrivial oriented unramified quotients. The complex plane $\C$ covers $\C^*$ and the complex tori $\C/\Gamma$, so these are Oka. The disc and its quotients are hyperbolic.
\end{proof}

A covering projection is a fiber bundle with a discrete fiber. Proposition \ref{OP-covering} generalizes to holomorphic fiber bundles with positive dimensional Oka fibers.

%
%
%    UP AND DOWN
%
%
\begin{theorem}
\label{OP-fiberbundles}
{\rm \cite[Theorem 5.5.4, p.\ 205]{FF:book}}
If $E$ and $X$ are complex manifolds and $\pi \colon E\to X$ is a holomorphic fiber bundle whose fiber $Y$ is Oka, then $X$ is Oka if and only if $E$ is Oka. In particular, if two of the manifolds $X$, $Y$, and $X\times Y$ are Oka, then so is the third.  
\end{theorem}

\begin{proof}[Sketch of proof]
Assume first that $E$ is an Oka manifold. Let $K\subset Q$ be compact convex sets in $\C^n$, and let $f\colon K\to X$ be a holomorphic map. By the homotopy lifting theorem, there exists a continuous lifting $h\colon K\to E$ of $f$. Since the fiber $Y$ of $\pi\colon E\to X$ is Oka, we can replace $h$ by a holomorphic lifting by using Corollary \ref{OP:lifting}. Since $E$ is Oka, we can approximate $h$ uniformly on $K$ by a holomorphic map $\wt h\colon Q\to E$. The map $\wt f=\pi\circ \wt h\colon Q \to X$ then approximates $f$ on $K$. This shows that $X$ is Oka. 

Conversely, assume that $X$ is Oka. Choose compact convex sets $K\subset Q\subset \C^n$ and a holomorphic map $h\colon K\to E$. Let $f=\pi\circ h\colon K\to X$. Since $X$ is Oka, we can approximate $f$, uniformly on $K$, by a holomorphic map $f_1 \colon Q \to X$. If the approximation is sufficiently close then we can also approximate $h$, uniformly on $K$, by a holomorphic map $h_1\colon K\to E$ that is a lifting of $f_1|_K$, in the sense that $\pi\circ h_1(x)= f_1(x)$ for all $x\in K$. (See \cite[p.\ 199]{FF:book} for the details.) Since $\pi\colon E\to X$ is a fiber bundle and the sets $K \subset Q$ are convex, the map $h_1$ extends to a continuous map $h_1 \colon Q \to E$ satisfying $\pi\circ h_1(x)= f_1(x)$ for all $x\in Q$. Finally, as the fiber of $\pi$ is an Oka manifold, Corollary \ref{OP:lifting} shows that $h_1$ can be deformed to a holomorphic lifting $\wt h\colon Q \to E$ of the map $f_1\colon Q \to X$ by a homotopy of liftings which remains uniformly close to $h_1$ on the set $K$. In particular, $\wt h$ approximates $h$ uniformly on $K$. Hence $E$ is an Oka manifold. 
\end{proof}

\begin{corollary}                                     
\label{P1-bundles}
If $\pi\colon E\to X$ is a holomorphic fiber bundle whose base and fiber are one of the Riemann surfaces $\C\P^1$, $\C$, $\C^*$, or a torus $\C/\Gamma$, then $E$ is an Oka manifold. In particular, all {\em minimal Hirzebruch surfaces} $H_l$ $(l=0,1,2,\ldots)$ are Oka manifolds. 
\end{corollary}

\begin{proof}
By Corollary \ref{p6.1} the list $\{\C\P^1, \C, \C^*, \C/\Gamma\}$ contains all Riemann surfaces which are Oka, so the first statement follows from Proposition \ref{OP-fiberbundles}. For the second one, note that minimal Hirzebruch surfaces are holomorphic $\C\P^1$-bundles over $\C\P^1$ \cite[p.\ 191]{Barth-Hulek}.
\end{proof}

%%%%%%%%%%%%%%%%%%%%%%%%%%%%%%%%%%%%%%%%%%%%%%%%%%%%%%%
%																											%		
%   COMPLEMENTS OF ALGEBRAIC SUBVARIETIES							%			
%																											%
%%%%%%%%%%%%%%%%%%%%%%%%%%%%%%%%%%%%%%%%%%%%%%%%%%%%%%%
	
\subsubsection{Complements of algebraic subvarieties}
\label{ss:complements}
We address the following question.

\smallskip
\noindent {\em Question:}
If $\wh Y$ is an algebraic Oka manifold and $A$ is a closed algebraic subvariety of $\wh Y$, when is the complement $Y=\wh Y\setminus A$ an Oka manifold?
\smallskip

The answer is negative in general if $A$ is a hypersurface; in such case the complement can even be Kobayashi hyperbolic. For example, Green's theorem \cite{Green} says that the complement of $2n+1$ hyperplanes in $\C\P^n$  in general position is hyperbolic. Kobayashi asked (\cite{Kobayashi1}, problem 3 on p.\ 132) whether the complement of a generic hypersurface of sufficiently high degree is hyperbolic. An affirmative answer was given for $n = 2$ by Siu and Yeung \cite{Siu-Yeung1996} in 1996, and in any dimension $n>1$ by Siu \cite{Siu2012} in 2012. 

Therefore it is natural to look for low degree hypersurfaces with Oka complements, and at algebraic subvarieties of lower codimension. The work on the first question was started recently by A.\ Hanysz \cite{Hanysz1}; see \S\ref{ss:hypersurfaces} below. In this section we show that any algebraic subvariety of codimension $>1$ in a Euclidean space, or in a complex Grassmannian, has Oka complement. 

%
%
%  COMPLEMENTS
%
%
\begin{proposition}
\label{complements}
{\em \cite[Proposition 5.5.8, p.\ 201]{FF:book}}
Let $\wh  Y$ denote one of the manifolds $\C^k$, $\C\P^k$, or a complex Grassmannian. If $A$ is an algebraic subvariety of $\wh Y$ of complex codimension $>1$, then the complement $Y=\wh Y\setminus A$ is an Oka manifold.        
\end{proposition}

\begin{proof}[Sketch of proof] We describe the main idea in the case $\wh Y=\C^k$. Suppose that $K$ is a compact convex set in $\C^n$ and $f\colon K\to Y=\C^k\setminus A$ is a holomorphic map. Approximate $f$ uniformly on a neighborhood of $K$ by a polynomial map $P\colon \C^n\to\C^k$. For a generic choice of $P$, the preimage $\Sigma:=P^{-1}(A)$ is an algebraic subvariety of codimension $>1$ in $\C^n$. The key point now is that there exists a Fatou-Bieberbach map $\phi\colon \C^n \to \C^n\setminus \Sigma$ which is close to the identity map on $K$ \cite[Corollary 4.12.2, p.\ 144]{FF:book}. Hence $P\circ\phi\colon \C^n\to \C^k$ is an entire map with range in $Y=\C^k\setminus A$ which approximates $f$ on $K$. 

A similar construction works for projective spaces and Grassmannians. Denote by $\pi\colon \C^{k+1}_* \to\C\P^k$ the standard quotient projection. A holomorphic map $f\colon K\to \C\P^k\setminus A$ from a compact convex set $K\subset \C^n$ lifts to a holomorphic map $g\colon K\to \C^{k+1}_*\setminus \pi^{-1}(A)$, and by the above argument we can find an entire approximation $\wt g\colon \C^n \to \C^{k+1}_*\setminus \pi^{-1}(A)$. The composition $\pi\circ \wt g\colon \C^n\to\C\P^k\setminus A$ is then an entire map approximating $f$ on $K$. 
\end{proof}

\begin{corollary}
\label{Hopf}
{\rm \cite[Corollary 5.5.9, p.\ 202]{FF:book}}
Every minimal Hopf manifold is Oka.   
\end{corollary}

\begin{proof}
A minimal Hopf manifold is an unramified holomorphic quotient of $\C^n\setminus\{0\}$ for some $n>1$. Since $\C^n\setminus\{0\}$ is an Oka manifold by Proposition \ref{complements}, the conclusion follows from Proposition \ref{OP-covering}.
\end{proof}

Proposition \ref{complements} fails in general for non-algebraic subvarieties of $\C^n$ irrespective of their dimension. Indeed, Rosay and Rudin \cite{Rosay-Rudin} found a discrete set in $\C^n$ whose complement is volume hyperbolic, so it fails to be Oka. (See \cite[Theorem 4.7.2]{FF:book}. Another example can be found on p.\ 202 in \cite{FF:book}.) However, complements of {\em tame} analytic subvarieties of codimension $>1$ in $\C^n$ are Oka; see Proposition \ref{tame-elliptic} below.

%%%%%%%%%%%%%%%%%%%%%%%%%%%%%%%%%%%%%%%%%%%%%%%%%%%%%%%
%																											%		
%   COMPLEMENTS OF HYPERSURFACES   										%			
%																											%
%%%%%%%%%%%%%%%%%%%%%%%%%%%%%%%%%%%%%%%%%%%%%%%%%%%%%%%
\subsubsection{Complements of hypersurfaces}
\label{ss:hypersurfaces}
We have already mentioned in \S\ref{ss:complements} that the complement of a generic algebraic hypersurface of high degree in $\C^n$ or $\C\P^n$ is hyperbolic. Hence it is natural to look for hypersurfaces of low degree with Oka complements. Examples of this type have recently been studied by A.\ Hanysz \cite{Hanysz1}. 

His first result concerns hyperplane arrangements in $\C\P^n$. Let $F_1,\ldots,F_N$ be nonzero homogeneous linear forms in $n + 1$ complex variables. We say that the hyperplanes in $\C\P^n$ defined by the equations $F_j =0$, $j = 1,\ldots,N$, are {\em in general position} if every subset of $F_1,\ldots,F_N$ of size at most $n + 1$ is linearly independent. If $N \le n + 1$, then a set of $N$ hyperplanes is in general position if and only if coordinates on $\C\P^n$ 
can be chosen so that these are the coordinate hyperplanes $z_j = 0$, $j = 0,\ldots,N-1$. By Green's theorem \cite{Green} the complement of at least $2n + 1$ hyperplanes in general position is hyperbolic, and the complement of a collection of at most $2n$ hyperplanes is never hyperbolic. For hyperplanes not in general position, some necessary conditions for hyperbolicity of the complement are known; see \cite[\S 3.10]{Kobayashi2}.

\begin{theorem}
\label{th:hyperplanes}
{\rm \cite[Theorem 3.1]{Hanysz1}}
Let $H_1,\ldots,H_N$ be distinct hyperplanes in $\C\P^n$. Then the complement $X= \C\P^n\setminus \bigcup_{j=1}^N H_j$ is Oka if and only if the hyperplanes $H_j$ are in general position and $N\le n + 1$. Furthermore, if such $X$ is not Oka then it is also not dominable by $\C^n$ and not $\C$-connected.
\end{theorem}

Hanysz's second result, which concerns complements of certain meromorphic graphs, is motivated by the theorem of Buzzard and Lu \cite[Proposition 5.1]{Buzzard-Lu} that the complement in $\C\P^2$ of a smooth cubic curve is dominable by $\C^2$. Their method was to construct a meromorphic function associated with the cubic, and a branched covering map from the complement of the graph of that function in $\C\times \C$ onto the complement of the cubic, and then show that the graph complement is dominable. Hanysz proved that the graph complement is in fact an Oka manifold. His result applies to meromorphic functions on Oka manifolds other than $\C$, subject to an additional hypothesis. 

\begin{theorem}
\label{th:merograph}
{\rm \cite[Theorem 4.6]{Hanysz1}}
Let $X$ be a complex manifold, and let $m\colon X \to \C\P^1$ be a holomorphic map. Denote by $G_m$ the graph of $m$. Suppose that $m$ can be written in the form $m = f + \frac{1}{g}$ for holomorphic functions $f$ and $g$. Then $(X \times \C)\setminus G_m$ is Oka if and only if $X$ is Oka.
\end{theorem}

Unfortunately this is not enough to settle the question of whether the complement of a smooth cubic in $\C\P^2$ is Oka, since it is not known whether the Oka property passes down to a (finite) ramified holomorphic quotient.

As Hanysz remarked, the existence of the decomposition $m = f + \frac{1}{g}$ is equivalent to the condition that the projection map $(X \times \C)\setminus G_m \to X$ admits a global holomorphic section. This projection is an elliptic submersion (Def.\ \ref{def:elsub} below). (It is easy to see that it is a stratified elliptic submersion, but Hanysz showed that there is no need to stratify the base \cite[Remark 4.12]{Hanysz1}.) However, unless either $m$ has no poles or $m\equiv \infty$, the projection is not an Oka map (see \S\ref{ss:Okamaps}) because it is not a topological fibration.

Another direction started by Hanysz in \cite{Hanysz2} is to consider Oka properties of the space  $X=\cO(\C\P^1,\C\P^1)$ of holomorphic self-maps of the Riemann sphere. It is well known that $X=\bigcup_{d=0}^\infty X_d$, where $X_d$ is the connected component of $X$ consisting of all rational maps $f(z)=p(z)/q(z)$ of degree precisely $d$. Is $X_d$ Oka for every $d$? For $d=0,1,2$, $X_d$ is complex homogeneous, and hence Oka in view of Theorem \ref{CAP:Lie}. For $d=3$, Hanysz proved that $X_3$ is strongly dominable and strongly $\C$-connected (Theorems 1.7 and 1.8 in \cite{Hanysz2}). The question whether $X_d$ for $d\ge 3$ is Oka remains open.

%%%%%%%%%%%%%%%%%%%%%%%%%%%%%%%%%%%%%%%%%%%%%%%%%%%%%%%
%																											%		
%   TORIC VARIETIES         													%			
%																											%
%%%%%%%%%%%%%%%%%%%%%%%%%%%%%%%%%%%%%%%%%%%%%%%%%%%%%%%
\subsubsection{Toric varieties}
\label{ss:toric}
The following result is due to F.\ L\'arusson. 

\begin{theorem}
\label{th:toric}
{\rm  \cite[Theorem 3]{Larusson-toric}}
Every smooth complex toric variety is Oka. 
\end{theorem}

\begin{proof}
Let $X$ be a smooth toric variety over $\C$.  If $X$ has a torus factor, say $X$ is isomorphic to $Y\times(\C^*)^k$, where $k\geq 1$ and $Y$ is another smooth toric variety, then, since $(\C^*)^k$ is Oka, Theorem \ref{OP-fiberbundles} shows that $X$ is Oka if and only if $Y$ is Oka.  

Hence we may assume that $X$ has no torus factor, so the construction in \cite[\S 5.1]{Cox-Little-Schenck} applies; see in particular Theorem 5.1.11.  We can write $X$ as a geometric quotient
\[ 
		X = (\C^m\setminus Z)/G, 
\]
where $Z$ is a union of coordinate subspaces of $\C^m$, and $G$ is a complex subgroup of $(\C^*)^m$ acting on $\C^m\setminus Z$ by diagonal matrices.  In fact, $G$ is isomorphic to the product of a torus and a finite abelian group \cite[Lemma 5.1.1]{Cox-Little-Schenck}, so $G$ is reductive.  Furthermore, $\codim Z\geq 2$ (\cite{Cox-Little-Schenck}, Exercise 5.1.13), so $\C^m\setminus Z$ is Oka by Proposition \ref{complements}. Since $X$ is smooth, $G$ acts freely on $\C^m\setminus Z$ \cite[Exercise 5.1.11]{Cox-Little-Schenck}.  

We claim that the projection $\C^m\setminus Z\to X$ is a holomorphic fiber bundle; since the fiber $G$ is Oka by Theorem \ref{CAP:Lie}, this will imply that $X$ is Oka in view of Theorem \ref{OP-fiberbundles}.

Note that $Z$, being a union of coordinate subspaces, is the intersection of unions of coordinate hyperplanes.  Thus $\C^m\setminus Z$ is the union of Zariski-open sets of the form $U=\C^m\setminus (H_1\cup\cdots\cup H_k)$, where $H_1,\dots,H_k$ are coordinate hyperplanes.  Each $U$ is affine algebraic (hence Stein), as well as $G$-invariant.  By slice theory for actions of reductive groups, the quotient map $U\to U/G$ is a holomorphic fiber bundle \cite[Corollary 5.5]{Snow}, or, from the algebraic point of view, a locally trivial fibration in the \'etale sense \cite[Corollaire 5]{Luna}.  It follows that $\C^m\setminus Z\to X$ is a holomorphic fiber bundle.
\end{proof}

\begin{remark}
(a) Every toric variety is birationally equivalent to complex projective space, which is Oka.  Thus, if we knew that the Oka property was birationally invariant, Theorem \ref{th:toric} would be immediate.  At present, it is not even known how the Oka property behaves with respect to blowing up a point, except in some special cases (see the results in \cite[pp.\ 252--255]{FF:book}). 

(b) It follows from \cite[Theorem 2.1]{AKZ}, which is more difficult to prove than Theorem \ref{th:toric}, 
that the smooth part of an {\em affine} toric variety is elliptic, and hence Oka (cf. Examples \ref{example:elliptic} (B) and \ref{ex:flexible} below). Smooth affine toric varieties are of the form $\C^n \times (\C^*)^m$, so they are obviously Oka.  
\end{remark}

%%%%%%%%%%%%%%%%%%%%%%%%%%%%%%%%%%%%%%%%%%%%%%%%%%%%%%%
%																											%		
%   ELLIPTIC AND SUBELLIPTIC MANIFOLDS								%			
%																											%
%%%%%%%%%%%%%%%%%%%%%%%%%%%%%%%%%%%%%%%%%%%%%%%%%%%%%%%

\subsection{Elliptic and subelliptic manifolds}
\label{ss:elliptic}
The notion of a {\em dominating spray} and of an {\em elliptic manifold} was introduced in Oka theory by M.\ Gromov \cite{Gromov:OP}. The main reference for this subsection is \cite[\S 5.15]{FF:book}.

%
%
%
%  		SPRAYS
%
%
\begin{definition}
\label{def:spray}
A {\em (global, holomorphic) spray} on a complex manifold $Y$ is a triple $(E,\pi,s)$ consisting of a holomorphic vector bundle $\pi \colon E\to Y$ {\rm (a spray bundle)} and a holomorphic map $s\colon E\to Y$ {\rm (a spray map)} such that for each point $y\in Y$ we have $s(0_y)=y$. The spray $(E,\pi,s)$ is {\em dominating} if the differential $ds_{0_y}\colon T_{0_y}E \to T_y Y$ maps the vertical subspace $E_y$ of $T_{0_y}E$ surjectively onto $T_y Y$ for every $y\in Y$. A complex manifold $Y$ is said to be {\em elliptic} if it admits a dominating holomorphic spray.
\end{definition}

A dominating spray can be viewed as a family of dominating holomorphic maps $s_y\colon E_y\cong \C^k\to Y$, $s_y(0)=y$, depending holomorphically on the point $y\in Y$.

A spray $(E,\pi,s)$ is algebraic if $\pi\colon E\to Y$ is an algebraic vector bundle over an algebraic manifold $Y$ and $s\colon E\to Y$ is an algebraic map. An algebraic manifold with a dominating algebraic spray is said to be {\em algebraically elliptic}. 

A relaxation of these conditions was introduced in \cite{FF:Subelliptic}. A finite family of sprays $(E_j,\pi_j,s_j)$ on $Y$ $(j=1,\ldots,m)$  is {\em dominating} if for every point $y\in Y$ we have
\begin{equation}
\label{eqn:domination}
    (d s_1)_{0_y}(E_{1,y}) + (d s_2)_{0_y}(E_{2,y})+\cdots
                     + (d s_m)_{0_y}(E_{m,y})= T_y Y.
\end{equation}
A complex manifold $Y$ is {\em subelliptic} if it admits a finite dominating family of sprays. An algebraic manifold $Y$ is {\em algebraically subelliptic} if it admits a finite dominating family of algebraic sprays.

By definition every elliptic manifold is also subelliptic, but the converse is not known in general. In particular, it is not known whether the (algebraically) subelliptic manifolds furnished by Corollary \ref{cor:subel-compl} (in projective spaces and Grassmannians) are all elliptic. The problem is that there is no general procedure for composing a dominating family of sprays into a single dominating spray. However, this can be done on a Stein manifold, and hence a subelliptic Stein manifold is elliptic (see \cite[Lemma 6.3.3, p.\ 246]{FF:book}). 

The following examples of dominating sprays were pointed out by Gromov \cite{Gromov:OP}.

\begin{example}
\label{example:elliptic}
{\rm \cite[Example 5.5.13]{FF:book}}

\noindent (A) {\em Every complex homogeneous manifold is elliptic.}  Indeed, assume that a complex Lie group $G$ acts on a complex manifold $Y$ transitively by holomorphic automorphisms of $Y$. Let $\ggoth\cong\C^p$ denote the Lie algebra of $G$ and
$\exp\colon \ggoth\to G$ the exponential map. The holomorphic map $s\colon Y\! \times\ggoth \cong Y\times\C^p \to Y$,
\[
      s(y,v)= \exp v\,\cdotp y \in Y, \quad  y\in Y,\ v\in \ggoth,
\]
is a dominating holomorphic spray on $Y$.

\smallskip
\noindent (B)
{\em If the tangent bundle of a complex manifold $Y$ is spanned by finitely many $\C$-complete holomorphic vector fields, then $Y$ is elliptic.} Indeed, let $V_1,\ldots,  V_m$ be $\C$-complete holomorphic vector fields on $Y$. Denote by 
$\phi_j^t(y)$ the flow of $V_j$. The map $s\colon Y\times \C^m \to Y$, given by
\begin{equation}
\label{flow-spray}
      s(y,t)= s(y,t_1,\ldots,t_m) =
      \phi_1^{t_1}\circ \phi_2^{t_2}\circ\cdots\circ \phi_m^{t_m}(y),
\end{equation}
satisfies $s(y,0)=y$ and $\frac{\di}{\di t_j}\, s(y,0)=  V_j(y)$ for all $y\in Y$ and $j=1,\ldots,m$. Thus $s$ is dominating at the point $y\in Y$ precisely when the vectors $V_1(y),\ldots,  V_m(y)$ span $T_y Y$. 
\end{example}

\begin{example}[\bf Holomorphically flexible manifolds] 
\label{ex:flexible} 
Example \ref{example:elliptic} (B) initiated a sub\-stan\-tial amo\-unt of contemporary research. In the paper \cite{Arzhantsev2} by Arzhantsev et al., a complex manifold $Y$ is said to be {\em holomorphically flexible at a point} $y\in Y$ if the values at $y$ of $\C$-complete holomorphic vector fields on $Y$ span the tangent space $T_y Y$; the manifold $Y$ is said to be {\em holomorphically flexible} if it is holomorphically flexible at every point. Clearly a connected manifold $Y$ is holomorphically flexible if it is holomorphically flexible at one point $y_0\in Y$ and the holomorphic automorphism group $\mathrm{Aut}Y$ acts transitively on $Y$. In \cite{Arzhantsev2,AKZ} the authors also study the analogous notion of {\em algebraic flexibility} of affine algebraic varieties. A holomorphically flexible Stein manifold is elliptic (\cite[Lemma 4.1]{Kaliman-Kut4},  \cite{Arzhantsev2}), and hence an Oka manifold in view of Corollary \ref{subell-Oka} below. 

For further results and references concerning flexible manifolds we refer to Chap.\ 4 in \cite{FF:book} and to the recent survey paper \cite{Kaliman-Kut4} on Anders\'en-Lempert theory. 
\end{example}

The relevance of subellipticity is shown by the following homotopy version of the Oka-Weil theorem, due (for the class of elliptic manifolds) to Gromov \cite{Gromov:OP}. (See Theorem 6.6.1 in \cite[p.\ 263]{FF:book} for a more general version.) This immediately implies that a subelliptic manifold enjoys CAP, and hence is an Oka manifold (see Corollary \ref{subell-Oka} below).

%
%  h-Runge, basic version
%  Fsubell: 3.1 Theorem, p.534
%
%
\begin{theorem}
\label{h-Runge1}               
% {\rm (\cite{Gromov:OP}, \cite{FPrezelj:OP1}, \cite{FF:subelliptic})}
Assume that $Y$ is a subelliptic manifold. Let $X$ be a Stein space and $K$ a compact $\cO(X)$-convex set in $X$.
Given a homotopy of holomorphic maps $f_t\colon K \to Y$ $(t\in [0,1])$ such that $f_0$ extends to a holomorphic map $f_0\colon X\to Y$, there exists for every $\epsilon>0$ a homotopy of holomorphic maps $\wt f_t\colon X\to Y$ $(t\in [0,1])$ such that $\wt f_0=f_0$ and 
\[
    \sup\left\{ \dist\bigl(\wt f_t(x),f_t(x)\bigr) \colon x\in K,\ t\in [0,1]\right\} < \epsilon.
\]
\end{theorem}

\begin{proof}[Sketch of proof] For simplicity we restrict attention to the case when $Y$ is elliptic. Let $(E,\pi,s)$ be a dominating spray on $Y$. Set $Z=X \times Y$ and denote by $h\colon Z\to X$ the projection onto the first factor; maps $X\to Y$ are then identified with sections $X\to Z$ of $h$. Furthermore, the spray $(E,\pi,s)$ on $Y$ defines a fiber-dominating spray on $Z$ (see \S\ref{ss:ellipticsub} below for this notion), still denoted $(E,\pi,s)$, which is independent of the base variable.

By compactness of the set $\bigcup_{t\in[0,1]} f_t(K)\subset Z$ and the fiber-domination property of the spray we can find an open set $V$ in $X$, with $K\subset V\Subset U$, and numbers $0=t_0<t_1<\cdots<t_k=1$ such that for every $j=0,1,\ldots,k-1$, the homotopy $f_t$ for $t\in [t_j,t_{j+1}]$ lifts to a homotopy of holomorphic sections $\xi_t$ of the restricted bundle $E|_{f_{t_j}(V)}$ satisfying 
\[
    s\circ \xi_t \circ f_{t_j}(x) = f_t(x), \quad x\in V,\ t\in [t_j,t_{j+1}].
\]
In particular, the family $\{\xi_t\}_{t\in [0,t_1]}$ consists of holomorphic sections over $f_0(V)\subset f_0(X)$ of the holomorphic vector bundle $E|_{f_{0}(X)} \to f_{0}(X)$. By the Oka-Weil theorem we can approximate this family, uniformly on a neighborhood of the set $f_0(K)$ in $f_0(X)$, by a homotopy of global holomorphic sections $\wt \xi_t \colon f_0(X)\to E|_{f_0(X)}$ $(t\in [0,t_1])$ such that $\wt \xi_0$ equals the zero section. (We actually use a parametric version of the Oka-Weil theorem which follows from the standard result by applying a continuous partition of unity in the parameter.) The holomorphic sections $\wt f_t := s\circ\tilde \xi_t \circ f_0 \colon X\to Z$ $(t\in [0,t_1])$ are then close to the respective sections $f_t$ near the set $K$. In particular, we may assume that $\wt f_{t_1}(K)$ is contained in a Stein tubular neighborhood of $f_{t_1}(V)$ in $Z$, and hence we can connect these two sections by a homotopy of nearby holomorphic sections over a neighborhood of $K$. We thus get a homotopy of sections $\wt f_t \colon X\to Z$ $(t\in[0,1])$, with $\wt f_t$ close to $f_t$ near $K$, such that $\wt f_0=f_0$, $\wt f_t$ is holomorphic on $X$ for $t\in [0,t_1]$, and it is holomorphic in a neighborhood of $K$ for $t\in (t_1,1]$.

We now use $\wt f_{t_1}\colon X\to Z$ as the new base section and repeat the above procedure on the next interval $[t_1,t_2]$. Assuming as we may that $\wt f_{t_1}$ is sufficiently close to $f_{t_1}$ near $K$, we can lift sections $\wt f_{t}$ for $t\in[t_1,t_2]$ to sections $\wt \xi_t$ of the restricted bundle $E|_{\wt f_{t_1}(X)}$. Applying the Oka-Weil theorem as before we obtain a new homotopy of sections which are holomorphic on $X$ for $t\in [0,t_2]$, and on a neighborhood of $K$ for $t\in [t_2,1]$. The proof is completed after $k$ such steps.
\end{proof}

\begin{corollary}
\label{subell-Oka}
{\rm \cite[Corollary 5.5.12, p.\ 204]{FF:book}}
Every subelliptic manifold (in particular, every elliptic manifold) is an Oka manifold. 
\end{corollary}

\begin{proof}
Let $K$ be a compact convex set in $\C^n$ and let $f\colon U\to Y$ be a holomorphic map from an open convex neighborhood $U\subset \C^n$ of $K$. We may assume that $0\in K$. Let $f_t(z)=f(tz)$ for $z\in U$ and $t\in[0,1]$; this is a homotopy from the constant map $f_0(z)=f(0)\in Y$ $(z\in\C^n)$ to the map $f=f_1$. If $Y$ is subelliptic then by Theorem \ref{h-Runge1} $f$ can be approximated uniformly on $K$ by entire maps $\C^n\to Y$, so $Y$ enjoys CAP. 
\end{proof}

Consider $\C^n$ as a subset of $\C\P^n$. A closed analytic subvariety $A\subset \C^n$ is said to be {\em tame} if its closure $\overline A\subset \C\P^n$ does not contain the hyperplane at infinity. A tame complex hypersurface is necessarily algebraic.

\begin{proposition} \label{tame-elliptic}
{\rm \cite[Proposition 5.5.14, p.\ 205]{FF:book}}
If $A\subset \C^n$ is a tame analytic subvariety with $\dim A\le n-2$, then $\C^n\setminus A$ is elliptic; if $A$ is algebraic, then $\C^n\setminus A$ is algebraically elliptic. In particular, every such manifold is Oka. 		
\end{proposition}

\begin{proof}
By Proposition 4.11.7 in \cite[p.\ 141]{FF:book} there exist finitely many $\C$-complete holomorphic vector fields $V_1,\ldots,V_m$ on $\C^n$ that vanish on $A$ and span the tangent space $T_z \C^n$ at every point $z\in \C^n\setminus A$. The associated spray $s\colon \C^n\times\C^m \to \C^n$ (\ref{flow-spray}) satisfies $s(z,t)= z$ for $z\in A$, $s(z,t)\in \C^n\setminus A$ for $z\in \C^n\setminus A$ and $t\in\C^m$, and the restriction $s\colon (\C^n\setminus A)\times\C^m \to \C^n\setminus A$ is a dominating spray over $\C^n\setminus A$. If $A$ is algebraic, then this construction gives an algebraic dominating spray on $\C^n\setminus A$. 
\end{proof}

\begin{corollary} 
\label{p6.3} 
%
% Corollary 1.5 (ii) in \cite{F6}
%
{\rm \cite[Corollary 5.5.15, p.\ 205]{FF:book}}
Let $X=\C^n/\Gamma$, where $\Gamma$ is a lattice in $\C^n$ $(n\ge 2)$. Then the complement $Y=X\setminus\{x_1,\ldots, x_m\}$ of any finite set of points in $X$ is an Oka manifold. 
\end{corollary}

\begin{proof}
Let $\pi\colon \C^n\to X=\C^n/\Gamma$ denote the quotient projection. Choose points $q_j\in \C^n$ such that $\pi(q_j)=x_j$ for $j=1,\ldots,m$. The discrete set $\Gamma_0 = \cup_{j=1}^m(\Gamma + q_j)\subset \C^n$ is tame (\cite{Buzzard}, \cite[Proposition 4.1]{Buzzard-Lu}), so $\C^p\setminus \Gamma_0$ is elliptic (and hence Oka) by Proposition \ref{tame-elliptic}. Since $\pi\colon \C^p\setminus \Gamma_0\to Y$ is an unramified covering, $Y$ is Oka by Proposition  \ref{OP-covering}.
\end{proof}

An interesting feature of algebraic subellipticity is that it can be localized as follows.
No comparable result is known in the holomorphic category.

%
%  Localization of subellipticity
%  Fsubell: 1.3 Proposition:
%  \cite[3.5.B., 3.5.C.]{Gromov:OP} and%\cite[Proposition 1.3]{FF:subelliptic}).
%
\begin{proposition}
\label{localization-subell}
{\rm \cite[Proposition 6.4.2, p.\ 251]{FF:book}}
If $Y$ is a quasi-projective algebraic manifold such that each point $y_0\in Y$ has a Zariski open neighborhood $U\subset Y$
and algebraic sprays $s_j\colon E_j\to Y$ $(j=1,2,\ldots,m)$, defined on algebraic vector bundles $p_j\colon E_j\to U$ and satisfying the domination property (\ref{eqn:domination}) 
%\[ 
%    (d s_1)_{0_y}(E_{1,y}) + (d s_2)_{0_y}(E_{2,y})\cdots + (d s_k)_{0_y}(E_{k,y})= T_y Y
%\] 
for every point $y\in U$, then $Y$ is algebraically subelliptic.   
\end{proposition}

\begin{corollary}
\label{cor:subel-compl}
If $A$ is an algebraic subvariety of codimension $>1$ in a projective space or a complex Grassmannian $X$, then $X\setminus A$ is algebraically subelliptic.
\end{corollary}

If $\wh Y\to Y$ is a holomorphic covering map of complex manifolds and $Y$ is elliptic, then it is easily seen that $\wh Y$ is elliptic.  It is not known whether ellipticity of $\wh Y$ implies ellipticity of $Y$ in general (compare with Proposition \ref{OP-covering} for Oka manifolds), but the following special case was proved recently by T.\ Ritter \cite{Ritter2}. 

\begin{theorem}
\label{thm_affinequotients}
If\ $\,\Gamma \subset \mathrm{Aff}(\C^n)$ is a discrete group of affine automorphisms of $\C^n$ acting freely and properly discontinuously on $\C^n$, then the quotient $\C^n/\Gamma$ is elliptic.
\end{theorem}

\begin{proof}
Each $\gamma \in \Gamma$ is of the form $\gamma(z) = Az + b$ where $A \in GL_n(\C)$ and $b\in \C^n$. Define $\lambda_\gamma = A$. The map $\sigma_z(w) = z + w$ is a dominating spray on $\C^n$, defined on the trivial bundle $\C^n\times\C^n$ over $\C^n$. A calculation gives
\begin{eqnarray*}
	\sigma_{\gamma(z)} \circ \lambda_\gamma \circ \sigma_z^{-1}(w) &=& \sigma_{\gamma(z)} \circ \lambda_\gamma(-z + w) = \sigma_{\gamma(z)}(-Az + Aw) \\&=& -Az + Aw + Az + b = \gamma(w)
\end{eqnarray*}
for every $w\in\C^n$; hence $\sigma_{\gamma(z)} \circ \lambda_\gamma \circ \sigma_z^{-1}= \gamma$ for every $z\in\C^n$. This condition means that the spray $\sigma_z(w)$ on $\C^n$ is $\Gamma$-equivariant. Hence $\sigma$ descends to the quotient
of the trivial bundle (which is the holomorphic tangent bundle of the quotient manifold $\C^n/\Gamma$) and induces a dominating spray on $\C^n/\Gamma$.
\end{proof}

%%%%%%%%%%%%%%%%%%%%%%%%%%%%%%%%%%%%%%%%%%%%%%%%%%%%%%%
%																											%		
%   BALL COMPLEMENTS			  													%			
%																											%
%%%%%%%%%%%%%%%%%%%%%%%%%%%%%%%%%%%%%%%%%%%%%%%%%%%%%%%
\subsection{Ball complements}
\label{ss:ball}
By the discussion in the previous subsection we have the following inclusions:
\[
	\{\rm elliptic\ manifolds\}\subset \{\rm subelliptic\ manifolds\}\subset \{\rm Oka\ manifolds\}.
\]
It follows directly from the respective definitions that we also have inclusions
\[
	\{\rm Oka\ manifolds\} \subset \{\rm strongly\ dominable\ manifolds\} \subset \{\rm dominable\ manifolds\}.
\]
For the notion of a (strongly) dominable manifold see \S \ref{ss:flex}.

\begin{problem}
\label{rem:Oka-elliptic}
Which of the above inclusions are proper? In particular, is every Oka manifold subelliptic? Is every subelliptic manifold also elliptic?
\end{problem}

It has recently been  shown by Andrist and Wold \cite{Andrist-Wold} that for every $n\ge 3$ the complement
$Y=\C^n\setminus\overline \B$ of the closed ball in $\C^n$ fails to be subelliptic. The main reason is
that every holomorphic vector bundle on $Y$ extends as a holomorphic vector bundle across
most of the boundary points of the ball (indeed, by Siu \cite{Siu1974}
it extends as a coherent analytic sheaf to all of $\C^n$), and hence any spray
also extends to a neighborhood of some point $p\in b\B$ in view of Hartogs' theorem.  
The analysis of the subellipticity condition in a neighborhood of any
such point easily leads to a contradiction. 

On the other hand, $\C^n\setminus \overline \B$ is a union of Fatou-Bieberbach domains \cite{Rosay-Rudin}, 
hence is strongly dominable. This implies the following corollary.

\begin{corollary} {\rm \cite{Andrist-Wold}}  
\label{cor:incl}
At least one of the following inclusions is proper:
\[
	\{\rm subelliptic\ manifolds\}\subset \{\rm Oka\ manifolds\} 
	\subset \{\rm strongly\ dominable\ manifolds\}.
\]
\end{corollary}

% The following question remains open.

\begin{problem}
\label{pr:ballcompl}
Let $n>1$. Is $\C^n\setminus \overline \B$ an Oka manifold?
\end{problem}

Partial positive results in in this direction have recently been found by Forstneri\v c and Ritter \cite{Forstneric-Ritter} who in particular proved the following.

\begin{theorem}
\label{th:FRitter}
If $L$ is a compact convex set in $\C^n$ for some $n>1$ then the complement $Y=\C^n\setminus L$ 
satisfies the basic Oka property with approximation and interpolation (BOPAI) for maps
from Stein manifolds of dimension $<n$ to $Y$.    
\end{theorem}

%%%%%%%%%%%%%%%%%%%%%%%%%%%%%%%%%%%%%%%%%%%%%%%%%%%%%%%
%																											%		
%   GOOD MANIFOLD					  													%			
%																											%
%%%%%%%%%%%%%%%%%%%%%%%%%%%%%%%%%%%%%%%%%%%%%%%%%%%%%%%
\subsection{Good manifolds}
\label{ss:good}
As was said above, it is not known whether every Oka manifold is elliptic or 
subelliptic. We mention a few existing partial results in this direction.
The following was proved by Gromov \cite{Gromov:OP} 
(see also \cite[Proposition 5.15.2, p.\ 237]{FF:book}). 

\begin{proposition}
\label{prop:SteinOkaelliptic}
{\rm \cite{Gromov:OP}}  Every Stein Oka manifold is elliptic.
\end{proposition}

\begin{proof} 
Let $Y$ be a Stein manifold. We embed $Y$ as the zero section in its holomorphic tangent bundle
$TY$. By Grauert's tubular neighborhood theorem there exist an open neighborhood 
$\Omega_0\subset TY$ of $Y$ and a holomorphic retraction $s_0\colon \Omega_0\to Y$.
In particular, the restriction of $s_0$ to $Y$ is the identity map on $Y$. 
Clearly there is a continuous map $\tilde s\colon TY\to Y$ which agrees with $s_0$ on a smaller 
neighborhood $\Omega\subset \Omega_0$ of $Y$ in $TY$. 
Note that $TY$ is also a Stein manifold, so if $Y$ is Oka then 
(using the jet interpolation property for holomorphic maps from Stein manifolds to $Y$)
there is a holomorphic map $s\colon TY\to Y$ which agrees with $s_0$ to the second order
along the zero section $Y\subset TY$. But such $s$ is a dominating holomorphic spray on $Y$,
so $Y$ is elliptic. 
\end{proof}

Gromov's result has been generalized by L\'arusson \cite{Larusson3} to a much larger class of manifolds 
(see also \cite[\S 3]{Larusson5} and \cite[p.\ 24]{FLflex}).  Call a complex manifold {\em good} if it is the image of an Oka map from a Stein manifold (see \S \ref{ss:Okamaps} for this notion), 
and {\it very good} if it carries a holomorphic affine bundle whose total space is Stein. 
The simplest examples of non-Stein good manifolds are the projective spaces.  Indeed, let $Q_n$ be the complement in $\P^n\times\P^n$ of the hypersurface
$	\big\{([z_0,\dots,z_n], [w_0,\dots,w_n]):z_0 w_0+\dots+z_n w_n=0\big\}$.
Note that $Q_n$ is the preimage of a hyperplane by the Segre embedding $\P^n\times \P^n\to \P^{n^2+2n}$, so $Q_n$ is Stein.  Let $\pi$ be the projection $Q_n\to\P^n$ onto the first component.  It is easily seen that $\pi$ has the structure of a holomorphic affine bundle with fibre $\C^n$, so $\P^n$ is very good.  (This observation is often called {\em the Jouanolou trick}.)

L\'arusson showed that the class of good manifolds contains all Stein manifolds and all quasi-projective manifolds; further, it is closed under taking submanifolds, products, covering spaces, finite branched covering spaces, and complements of analytic hypersurfaces.  The same is true of the class of very good manifolds. 

The following observation follows directly from Theorem \ref{OP-fiberbundles} and Proposition
\ref{prop:SteinOkaelliptic}.

\begin{corollary}
\label{cor:good}
A good manifold is Oka if and only if it is the image of an Oka map from an elliptic manifold.  
A very good manifold is Oka if and only if it carries an affine bundle whose total space is elliptic.
\end{corollary}

This is a purely geometric characterization of the Oka property that holds, for example, for all quasi-projective manifolds. 

\begin{problem}
Is every complex manifold good, or even very good?
\end{problem}

%%%%%%%%%%%%%%%%%%%%%%%%%%%%%%%%%%%%%%%%%%%%%%%%%%%%%%%
%																											%		
%   STRATIFIED OKA MANIFOLD  													%			
%																											%
%%%%%%%%%%%%%%%%%%%%%%%%%%%%%%%%%%%%%%%%%%%%%%%%%%%%%%%
\subsection{Stratified Oka manifolds}
\label{ss:stratOka}
The following notion was introduced in \cite{FLflex}.

\begin{definition}
\label{def:stratOka}
A complex manifold $Y$ is a \textit{stratified Oka manifold} if it admits a stratification $Y=Y_0\supset Y_1\supset\cdots\supset Y_m=\varnothing$ by closed complex subvarieties such that every connected component of each difference $Y_{j-1}\setminus Y_j$ $(j=1,\ldots,m)$ is an Oka manifold.
\end{definition}

While it is not known whether every stratified Oka manifold is actually an Oka manifold, we have the following result which we shall not prove here.

\begin{theorem}  \label{t:stratified-Oka}
{\rm \cite[Theorem 2]{FLflex}} A stratified Oka manifold is strongly dominable.
\end{theorem}

\begin{example} (See \cite[\S 3]{FLflex}.) 
We show that every {\em Kummer surface} $Y$ admits a stratification $Y\supset C\supset\varnothing$, where $C$ is the union of 16 mutually disjoint smooth rational curves and the difference $Y\setminus C$ is an Oka manifold \cite[Lemma 7]{FLflex}. Thus $Y$ is stratified Oka, and we have the following corollary to Theorem \ref{t:stratified-Oka}.

\begin{corollary}  \label{c:Kummer-strongly-dominable}
Every Kummer surface is strongly dominable.
\end{corollary}

Let us recall the structure of Kummer surfaces; see \cite{Barth-Hulek} for more information.  

Let $\T$ be a complex 2-torus, the quotient of $\C^2$ by a lattice $\Z^4\cong\Gamma\subset \C^2$ of rank $4$, acting on $\C^2$ by translations.  Let $\pi\colon \C^2\to\T=\C^2/\Gamma$ be the quotient map.  The involution $\C^2\to \C^2$, $(z_1,z_2)\mapsto (-z_1,-z_2)$, descends to an involution $\sigma:\T\to\T$ with precisely 16 fixed points $p_1,\ldots,p_{16}$.  In fact, if $\omega_1,\ldots,\omega_4\in\C^2$ are generators for $\Gamma$, then $p_1,\ldots,p_{16}$ are the images under $\pi$ of the 16 points $c_1\omega_1+\cdots+c_4\omega_4 \in\C^2$, where $c_1,\ldots,c_4\in\{0,\tfrac 1 2\}$; denote these points by $q_1,\ldots, q_{16}$. The quotient $\T/\{1,\sigma\}$ is a 2-dimensional complex space with 16 singular points $p'_1,\ldots,p'_{16}$ (the images of $p_1,\ldots,p_{16}$). Blowing up each $p'_j$ yields a smooth compact complex surface $Y$ containing 16 mutually disjoint smooth rational curves $C_1,\ldots,C_{16}$. This is the Kummer surface associated to the torus $\T$, or to the lattice $\Gamma$.

Here is an alternative description.  Let $X$ denote the surface obtained by blowing up the torus $\T$ at each of the 16 points $p_1,\ldots,p_{16}$.  Let $E_j\cong\P_1$ denote the exceptional divisor over $p_j$.  The involution $\sigma$ of $\T$ lifts to an involution $\tau: X \to X$ with the fixed point set $E=E_1\cup\cdots\cup E_{16}$.  The eigenvalues of the differential $d\tau$ at any point of $E$ are $1$ and $-1$.  Hence the quotient $X/\{1,\tau\}$ is smooth and contains 16 rational $(-2)$-curves $C_j\cong\P_1$, the images of the rational $(-1)$-curves $E_j$ in $X$.  The quotient is the Kummer surface $Y$.  Denoting by $\widehat \C^2$ the surface obtained by blowing up $\C^2$ at every point of the discrete set $\widetilde \Gamma= \bigcup\limits_{j=1}^{16}(q_j+\Gamma)$, we have the following diagram (see \cite[p.\ 224]{Barth-Hulek}):
\vskip -2mm
\[ 
	\xymatrix{ \widehat\C^2 \ar[r] \ar[d] & X \ar[r] \ar[d]  & Y \ar[d] \\ 
	\C^2 \ar[r]^{\pi} & \T \ar[r] & \T/\{1,\sigma\}} 
\]

The union $C=\cup_{j=1}^{16} C_j$ is Oka by Corollary \ref{p6.1}. The complement $Y \setminus C$ is biholomorphic to $\T\setminus\{p_1,\ldots,p_{16}\}$ which is Oka by Corollary \ref{p6.3}. This shows that the Kummer surface $Y$ is a stratified Oka manifold.
\qed\end{example}

Kummer surfaces are dense in the moduli space of all K3 surfaces, but we do not know whether it follows that all K3 surfaces are strongly dominable.  In fact, strong dominability is in general not closed in families of compact complex manifolds (see Corollary \ref{c:not-closed} below).

%%%%%%%%%%%%%%%%%%%%%%%%%%%%%%%%%%%%%%%%%%%%%%%%%%%%%%%
%																											%		
%   The OKA PROPERTY FOR COMPACT COMPLEX SURFACES 	  %			
%																											%
%%%%%%%%%%%%%%%%%%%%%%%%%%%%%%%%%%%%%%%%%%%%%%%%%%%%%%%
\subsection{The Oka property for compact complex surfaces}
\label{ss:Class7}
The information in this subsection is taken from \cite{FLflex}; for proofs we refer to the original source. 

Here is what we know about which minimal compact complex surfaces are Oka (see \cite[p.\ 4]{FLflex}); $\kappa\in\{-\infty,0,1,2\}$ denotes the Kodaira dimension. For the classification of surfaces see \cite{Barth-Hulek}.

$\kappa=-\infty$:  Rational surfaces are Oka.  A ruled surface is Oka if and only if its base is Oka.  Theorem~\ref{t:class-VII} belows covers surfaces of class VII if the global spherical shell conjecture is true.

$\kappa=0$:  Bielliptic surfaces, Kodaira surfaces, and tori are Oka.  It is unknown whether any or all K3 surfaces or Enriques surfaces are Oka.

$\kappa=1$:  Buzzard and Lu determined which properly elliptic surfaces are dominable \cite{Buzzard-Lu}.  Nothing further is known about the Oka property for these surfaces.

$\kappa=2$:  Surfaces of general type are not dominable (this is an easy consequence of \cite{Kobayashi-Ochiai}, Theorem~2), and hence not Oka.

Class VII in the Enriques-Kodaira classification comprises the non-algebraic compact complex surfaces of Kodaira dimension $\kappa=-\infty$.  Minimal surfaces of class VII fall into several mutually disjoint classes.  For second Betti number $b_2=0$, we have {\em Hopf surfaces} and {\em Inoue surfaces}.  For $b_2\geq 1$, there are {\em Enoki surfaces}, {\em Inoue-Hirzebruch surfaces}, and {\em intermediate surfaces}; together they form the class of {\em Kato surfaces}.  By the \textit{global spherical shell conjecture}, currently proved only for $b_2=1$ by Teleman \cite{Teleman}, every minimal surface of class VII with $b_2\geq 1$ is a Kato surface.  Assuming that the conjecture holds, it is possible to determine which minimal surfaces of class VII are Oka.

\begin{theorem}  \label{t:class-VII}
{\rm \cite[Theorem 4]{FLflex}}
Minimal Hopf surfaces and Enoki surfaces are Oka. Inoue surfaces, Inoue-Hirzebruch surfaces, and intermediate surfaces, minimal or blown up, are not strongly Liouville, and hence not Oka. 
\end{theorem}

The notion of a strongly Liouville manifold was introduced just before Corollary \ref{cor:Liouville}.

Enoki surfaces are generic among Kato surfaces.  Inoue-Hirzebruch surfaces and inter\-mediate surfaces can be obtained as degenerations of Enoki surfaces.  
Thus, Theorem \ref{t:class-VII} yields the following corollary.

\begin{corollary}  \label{c:not-closed}
{\rm \cite[Corollary 5]{FLflex}}  
Compact complex surfaces that are Oka can degenerate to a surface that is not strongly Liouville (and hence is not Oka).  Consequently, the following properties are in general not closed in holomorphic families of compact complex manifolds: (a) the Oka property; (b) the stratified Oka property; (c) strong dominability; (d) dominability; (e) $\C$-connectedness; (f) strong Liouvilleness.
\end{corollary}

In this context we mention the recent result by F.\ L\'arusson \cite{Larusson-deformations} that, in a holomorphic family of compact complex manifolds, the set of Oka fibers corresponds to a $G_\delta$ subset of the base manifold. It is an open question whether this set is always open. Corollary \ref{c:not-closed} says that the set of Oka fibers is not necessarily closed. In the same paper, L\'arusson gave a necessary and sufficient condition for the limit fiber of a sequence of Oka fibers to be Oka in terms of a newly introduced {\em uniform Oka property}. He also considered holomorphic submersions with noncompact fibers. 

Another general problem is to understand whether an Oka surface blown up at a point remains Oka; the corresponding question can be asked for blow-downs. For example, a complex torus (of any dimension $n>1$) blown up at finitely many points is Oka (see \cite[p.\ 255]{FF:book}). The proofs is obtained by showing that $\C^n$ (the universal covering space of the torus), blown up at all points of a discrete set $\widetilde \Gamma= \bigcup\limits_{j=1}^{m}(q_j+\Gamma)$ where $q_j\in \C^n$ and $\Gamma$ is a lattice in $\C^n$ (cf.\ \S \ref{ss:stratOka}), is Oka. Taking $\wt \Gamma$ to be the preimage of the blown-up points in our torus and applying Proposition \ref{OP-covering} we infer that the blown-up torus is Oka. We are unable to give a similar proof for the Hopf surfaces (quotients of $\C^2\setminus\{0\}$).

\begin{problem}
\label{pr:Hopf}  
Is a blown-up Hopf surface an Oka manifold?
\end{problem}

%%%%%%%%%%%%%%%%%%%%%%%%%%%%%%%%%%%%%%%%%%%%%%%%%%%%%%%%%%%%%%%%%%%%%%%
%																											 							  %		
%   GROMOV'S OKA PRINCIPLE FOR SECTIONS OF ELLIPTIC SUBMERSIONS  	    %			
%																										    							%
%%%%%%%%%%%%%%%%%%%%%%%%%%%%%%%%%%%%%%%%%%%%%%%%%%%%%%%%%%%%%%%%%%%%%%%

\subsection{Gromov's Oka principle for elliptic submersions}
\label{ss:ellipticsub}
We now present the most advanced known version of the Oka principle --- for sections of stratified subelliptic submersions over Stein spaces (cf.\ Theorem \ref{SES:OP} below). The notions and results in this subsection generalize those in \S \ref{ss:elliptic} above. For further results and complete 
proofs we refere to Chapter 6 in \cite{FF:book}.

Let $X$ and $Z$ be complex spaces and let $h\colon Z\to X$ be a holomorphic submersion. Given a point $z\in Z$, the subspace $VT_z Z= \Ker dh_z$ of the tangent space $T_z Z$ is called the {\em vertical tangent space} of $Z$ at $z$ (relative to $h$). 

A {\em fiber-spray} on $Z$ is a triple $(E,\pi,s)$, where $\pi\colon E\to Z$ is a holomorphic vector bundle and $s\colon E\to Z$ is a holomorphic map, such that for each point $z\in Z$ we have 
\[
    s(0_z)=z,\quad s(E_z) \subset Z_{h(z)} = h^{-1}(h(z)).
\]
The restriction of the differential $ds_{0_z}\colon T_{0_z}E \to T_z Z$ to the vertical subspace $E_z \subset T_{0_z}E$ maps $E_z$ to $VT_z Z$ and is called the {\em vertical derivative} of $s$ at the point $z\in Z$:
\begin{equation}
\label{vertical-derivative}
    Vd s_z = ds_{0_z}|_{E_z} \colon E_z \to VT_z Z.
\end{equation}
The spray is said to be {\em fiber-dominating} if the vertical derivative $Vd s\colon E\to VT Z$ is surjective.

Note that a (dominating) spray on a complex manifold $Y$ (see Def.\ \ref{def:spray} above) is the same thing as a (dominating) fiber-spray on the trivial submersion $Y\to X=\textrm{point}$. 

Similarly one defines dominability of a finite family of fiber-sprays.

\begin{definition}
\label{def:elsub}
A holomorphic submersion $h\colon Z\to X$ is said to be an {\em elliptic submersion} (resp.\ a {\em subelliptic submersion})
if every point $x_0\in X$ admits an open neighborhood $U\subset X$ such that the restricted submersion $h\colon Z|_U \to U$ admits a fiber-dominating spray (resp.\ a finite fiber-dominating family of fiber-sprays).

The submersion $h\colon Z\to X$ is {\em stratified elliptic (resp.\ stratified subelliptic)} if there exists a stratification of $X$ by closed complex subvarieties $X=X_0\supset X_{1}\supset\cdots\supset X_m=\varnothing$ such that every difference $S_k=X_k\bs X_{k+1}$ is nonsingular and the restricted submersion $h\colon Z|_{S_k} \to S_k$ is elliptic (resp.\ subelliptic) for every $k=0,\ldots,m-1$. 
\end{definition}

Examples of (sub-) elliptic submersions can be found in \cite{Gromov:OP} and in \cite[\S 6.4]{FF:book}.

The following is the most general known form of the Oka principle for sections of stratified holomorphic submersions over Stein spaces.

\begin{theorem} \label{SES:OP} 
{\rm \cite[Theorem 6.2.2, p.\ 243]{FF:book}} 
If $h\colon Z\to X$ is a stratified subelliptic submersion onto a  Stein space $X$, then sections $X\to Z$ satisfy the parametric Oka property with approximation and interpolation {\rm (POP)}. In particular, the inclusion $\Gamma_{\cO}(X,Z) \hra  \Gamma_{\cC}(X,Z)$ of the space of holomorphic sections into the space of continuous sections is a weak homotopy equivalence. 
\end{theorem}

The basic form of Theorem \ref{SES:OP} (for elliptic submersions onto Stein manifolds) is due to Gromov \cite{Gromov:OP}. The result as stated here was obtained in a series of papers; a complete proof, with references to the original sources, is available in \cite[Chapter 6]{FF:book}. This version of the Oka principle has already been used in some applications where the corresponding result for stratified fiber bundles (Theorem \ref{th:CAP} in \S \ref{ss:Oka}) does not suffice. In particular, the solution of the {\em holomorphic Vaserstein problem}, due to Ivarsson and Kutzschebauch  \cite{Ivarsson-Kutzschebauch2}, uses the Oka principle for sections of stratified elliptic submersions.

The fiber-domination property is used in a similar way as in the proof of Theorem \ref{h-Runge1} above. In the non-stratified case we use fiber-sprays on $Z$ over small open Stein sets $U\subset X$ to prove the exact analogue of Theorem \ref{h-Runge1}. That is, a homotopy of holomorphic sections over a compact $\cO(X)$-convex subset $K\subset U$, with the initial section holomorphic over $U$, can be approximated by a homotopy of holomorphic sections over $U$. The analogous result holds with continuous dependence on a parameter \cite[\S 6.6]{FF:book}. In the stratified case, this approximation result is used within strata. This replaces the role of CAP (see Def.\ \ref{def:CAP}) in the construction of global holomorphic sections. 

In fact, there is an axiomatic version of the approximation property for submersions, called the {\em (parametric) homotopy approximation property}, (P)HAP, which allows us to prove Theorem \ref{SES:OP} (cf.\ Theorem 6.6.6 in \cite[p.\ 266]{FF:book}).  

There are other key differences between the proof of the Oka principle for stratified fiber bundles (i.e., of Theorem \ref{th:CAP}) and in the stratified submersion case. While the first result is proved by inductive application of a one-step procedure, the proof of the latter uses an elaborate scheme of constructing a global section from a simplicial complex of local holomorphic sections and homotopies between them.

%%%%%%%%%%%%%%%%%%%%%%%%%%%%%%%%%%%%%%%%%%%%%%%%%%%%%%%%%
%                                                       %
%																								        %
%   OKA PROPERTIES OF HOLOMORPHIC MAPS                  %
%  		                                                  %
%																								        %
%																								        %
%%%%%%%%%%%%%%%%%%%%%%%%%%%%%%%%%%%%%%%%%%%%%%%%%%%%%%%%% 
%
\subsection{Oka maps}
\label{ss:Okamaps}
Following the philosophy of Grothendieck -- that a property should not be defined only for objects in a category, but also for morphisms -- we now extend the notion of the Oka property to holomorphic maps. This subsection is mainly based on \cite[\S 6.14]{FF:book} and on the papers \cite{FF:OkaMaps,FF:Invariance}. 

Assume that $\pi\colon E\to B$ is a holomorphic map of complex spaces, $X$ is a Stein space, $P_0 \subset P$ are compact Hausdorff spaces (the parameter spaces), and $f\colon P\times X\to B$ is an $X$-holomorphic map, meaning that $f(p,\cdotp)\colon X\to B$ is holomorphic on $X$ for every fixed $p\in P$. Consider maps in the following diagram 
(cf.\ \cite[p.\ 287]{FF:book}):
\[
	\xymatrix{ P_0\times X \ar[d]_{incl} \ar[r]^{\ \ \ f} & E \ar[d]^{\pi} \\ 
	           P\times X \ar[r]^{\ \ \ f} \ar[ur]^{F} & B }
\]
A map $F\colon P\times X\to E$ satisfying $\pi\circ F=f$ is said to be a {\em lifting} of $f$; such a map $F$ is {\em $X$-holomorphic} on $P_0$ if $F(p,\cdotp)$ is holomorphic for every $p\in P_0$.

%
%
%  DEFINITION OF POP FOR MAPS
%
%
\begin{definition}
\label{def:POPmap}
A holomorphic map $\pi\colon E\to B$ enjoys the {\em Parametric Oka Property} {\rm (POP)} if for any collection $(X,X',K,P,P_0,f,F_0)$, where $X$ is a  Stein space, $X'$ is a closed complex subvariety of $X$,
$P_0\subset P$ are compact Hausdorff spaces, $f\colon P\times X\to B$ is an $X$-holomorphic map, and $F_0 \colon P\times X\to E$ is a continuous map such that $\pi\circ F=f$, $F_0(p,\cdotp)$ is holomorphic on $X$ for all $p\in P_0$ and is holomorphic on $K\cup X'$ for all $p\in P$, there exists a homotopy $F_t\colon P\times X\to E$ satisfying the following for all $t\in [0,1]$:
\begin{itemize}
\item[\rm (i)]    $\pi\circ F_t=f$,
\item[\rm (ii)]   $F_t=F_0$ on $(P_0\times X) \cup (P\times X')$,
\item[\rm (iii)]  $F_t$ is $X$-holomorphic on $K$ and uniformly 
close to $F_0$ on $P\times K$, and
\item[\rm (iv)]   the map $F_1\colon P\times X\to E$ is $X$-holomorphic.
\end{itemize}
The map $\pi\colon E\to B$ enjoys the {\em Basic Oka Property} {\rm (BOP)} if the above holds for the case when $P$ is a singleton and $P_0=\varnothing$.
\end{definition}

Note that a complex manifold $E$ enjoys a certain Oka property if and only if the trivial map $E\to point$ does. 

POP of a map $\pi\colon E\to B$ is illustrated by the following diagram.
\vskip -3mm
\[
\xymatrix{ 
	P_0 \ar[d]_{incl} \ar[r] & \cO(X,E) \ar[d] \ar[r] 
	& \cC(X,E) \ar[d]^{\pi} \\
  P   \ar[r]_{\!\!\!\!\!\! f}  \ar[ur]^{F_1}	\ar[urr]_{\qquad F_0} 
  & \cO(X,B) \ar[r] & \cC(X,B) } 
\]

The problem of lifting a single holomorphic map $f\colon X\to B$ to a holomorphic map $F\colon X\to E$ easily reduces to the problem of finding a section of the pull-back map. As an example, we prove the following result \cite[Corollary 6.14.2]{FF:book}.

\begin{corollary}
\label{SES-BOP} 
{\rm (i)} Every stratified subelliptic submersion enjoys {\rm BOP}.    

\noindent {\rm (ii)} 
Every stratified holomorphic fiber bundle with Oka fibers enjoys {\rm BOP}. 
\end{corollary}

\begin{proof}
Let $\wt \pi \colon f^*E\to X$ denote the pull-back of a holomorphic submersion $\pi\colon E\to B$ by a holomorphic map $f\colon X\to B$. If $\pi\colon E\to B$ is a stratified subelliptic submersion, then so is $\wt \pi \colon f^*E \to X$. (Stratify $X$ such that each stratum is mapped by $f$ into a stratum of $B$ over which the submersion $\pi$ is subelliptic.) Then sections of $\wt \pi \colon f^*E \to X$ are in bijective correspondence with liftings $X\to E$ of $f$. Since sections of $f^*E \to X$ satisfy BOP by Theorem \ref{th:CAP}, $\pi$ also satisfies BOP.  The same proof applies in case (ii).
\end{proof}

The above argument fails in the parametric case since we do not have a holomorphic dependence of the pull-back $f^*E \to X$ on the map $f\colon X\to B$ in a given continuous family of maps. Nevertheless, the implication BOP$\,\Rightarrow\,$POP still holds.

%
%
%  BOP implies POP for maps!!!
%
%
\begin{theorem}
\label{mapsBOP-POP}
{\em (\cite{FF:OkaMaps}, \cite[Theorem 6.14.3]{FF:book})}
For every holomorphic submersion of complex spaces we have the implication $\hbox{\rm BOP}\Longrightarrow \hbox{\rm POP}$, provided that {\rm POP} is restricted to parameter spaces $P_0\subset P$ that are Euclidean compacta.
\end{theorem}

The following consequence of Corollary \ref{SES-BOP} and Theorem \ref{mapsBOP-POP} gives the most general known Oka property of holomorphic maps.

\begin{corollary}
\label{SES-POP}
{\rm (i)} Every stratified subelliptic submersion enjoys {\rm POP}. 

\noindent {\rm (ii)} Every stratified holomorphic fiber bundle with Oka fibers enjoys {\rm POP}. 
\end{corollary}

Another interesting point is that POP is a {\em local property} in the following sense. 

%
%
%  LOCALIZATION OF POP FOR MAPS
%
%
\begin{theorem}[\bf Localization principle for POP] 
{\rm \cite[Theorem 4.7]{FF:Invariance}}
\label{localization-of-POP}
A holomorphic submersion $\pi\colon E\to X$ of a complex space $E$ onto a complex space $X$ satisfies \POP\ if and only if
every point $x\in X$ admits an open neighborhood $U_x\subset X$ such that the restricted submersion $\pi\colon E|_{U_x}\to U_x$ satisfies \POP.
\end{theorem}

The proof of this result uses the fact that POP follows from PHAP, applied over small open subsets of the base space.

In analogy to the class of Oka manifolds, we introduce the class of {\em Oka maps} (\cite[\S 16]{Larusson2}, \cite{FF:OkaMaps}).

\begin{definition}
\label{def:Oka-map}
A holomorphic map $\pi\colon E\to B$ of complex spaces is said to be an {\em Oka map} if it is a Serre 
fibration and it enjoys \POP. 
\end{definition}

A holomorphic map is an Oka map precisely when it is an {\em intermediate fibration} in L\'arusson's model category %\cite{Larusson2,Larusson3} 
(see \S \ref{sec:Appendix}). The following is an immediate consequence of the definitions and of Theorem \ref{th:CAP}.

\begin{corollary}
\label{cor:up-down}
Assume that $E$ and $B$ are complex manifolds and $\pi\colon E\to B$ is a surjective Oka map. Then $E$ is an Oka manifold if and only if $B$ is an Oka manifold.
\end{corollary}

Corollary \ref{SES-POP} implies the following result.

\begin{corollary}
\label{cor:Okamaps} {\rm \cite[Corollary 6.14.8]{FF:book}}
{\rm (i)} A holomorphic fiber bundle projection is an Oka map if and only if the fiber is an Oka manifold.         

\noindent {\rm (ii)} A stratified subelliptic submersion, or a stratified holomorphic fiber bundle with Oka fibers, is an Oka map if and only if it is a Serre fibration.
\end{corollary}

%%%%%%%%%%%%%%%%%%%%%%%%%%%%%%%%%%%%%%%%%%%%%%%%%%%%%%%%%%%%%%%%%
%																																%
%   METHODS TO PROVE THE OKA PRINCIPLE   												%
%																																%
%%%%%%%%%%%%%%%%%%%%%%%%%%%%%%%%%%%%%%%%%%%%%%%%%%%%%%%%%%%%%%%%%	
\section{Methods to prove the Oka principle}
\label{sec:methods}
In this section we prove Theorem \ref{th:CAP} in the basic case, that is, for sections of holomorphic fiber bundles with CAP fibers over Stein manifolds, and without parameters. The general case is treated in \cite[Chap.\ 5]{FF:book}. 

We begin in \S\ref{ss:Stein} by recalling the properties of Stein manifolds that will be important for us. The monographs by Grauert and Remmert \cite{Grauert-Remmert}, by Gunning and Rossi \cite{Gunning-Rossi}, and by H\"ormander \cite{Hormander:SCV}, now considered classics, are still excellent sources for the theory of Stein manifolds and Stein spaces. Our exposition is influenced by \cite{FLsurvey} and \cite[Chap. 2]{FF:book}. We also mention a result of Eliashberg on constructing Stein structures on almost complex manifolds with a suitable handlebody structure, and we state the {\em soft Oka principle} (Theorem \ref{FS1:Main1}).

In \S \ref{ss:Cartanpair}--\S \ref{ss:sprays} we explain the technique of gluing holomorphic sprays of sections, a key ingredient in the proof. This method has already found numerous applications. The other main ingredient is the local analysis near a Morse critical point of a strongly plurisubharmonic exhaustion function; here we use the normal form  furnished by Lemma \ref{spsh-normal-form}, a local Mergelyan approximation theorem for manifold-valued maps, and a device for reducing the problem of extending a holomorphic section across a critical level to the noncritical case. Using these tools, we prove a basic version of Theorem \ref{th:CAP} in \S\ref{ss:proof}.

%%%%%%%%%%%%%%%%%%%%%%%%%%%%%%%%%%%%%%%%%%%%%%%%%%%%%%%%%%%%%%%%%
%																																%
%   STEIN MANIFOLDS                     												%
%																																%
%%%%%%%%%%%%%%%%%%%%%%%%%%%%%%%%%%%%%%%%%%%%%%%%%%%%%%%%%%%%%%%%%	
\subsection{Stein manifolds}
\label{ss:Stein}
The central concept of classical several complex variables is that of a {\em domain of holomorphy} --- a domain in $\C^n$ with a holomorphic function that does not extend holomorphically to any larger domain, even as a multivalued function.  By classical function theory, every domain in $\C$ is a domain of holomorphy.  One of the key discoveries that got several complex variables started at the turn of the 20th century is that this is far from true in higher dimensions.  

The notion of a {\em Stein manifold}, introduced by K.\ Stein in 1951 \cite{Stein1951} under the name of a {\em holomorphically complete complex manifold}, generalizes domains of holomorphy to the setting of manifolds.  There are at least four fundamentally different characterizations of Stein manifolds; the equivalence of any two of them is a nontrivial theorem. Stein's original definition, simplified by later developments, states that a complex manifold $X$ is {\it Stein} if it satisfies the following two axioms.
\begin{enumerate}
\item  Holomorphic functions on $X$ separate points: for any pair of distinct points $x,x'\in X$, there exists a holomorphic function $f\in\cO(X)$ with $f(x)\neq f(x')$.
\item  $X$ is {\em holomorphically convex}, that is, if $K\subset X$ is compact, then its $\cO(X)$-hull 
\[
	\wh K= \{x\in X\colon |f(x)| \le\sup_K |f|,\ \forall f\in\cO(X)\}
\] 
is also compact.  Equivalently, if $E\subset X$ is not relatively compact, then there is an $f\in\cO(X)$ such that $f|_E$ is unbounded.
\end{enumerate}

A domain in $\C^n$ is Stein if and only if it is a domain of holomorphy.  Every noncompact Riemann surface is Stein \cite{Behnke-Stein1}.  

Second, a connected complex manifold is Stein if and only if it is biholomorphic to a closed complex submanifold of $\C^m$ for some $m$.  Namely, submanifolds of $\C^m$ are clearly Stein. (The coordinate functions on $\C^m$, restricted to the  submanifold, satisfy the above definition. More generally, it is easy to see that a closed complex submanifold of a Stein manifold is itself Stein.)  Conversely, R.\ Remmert proved in 1956 that every connected Stein manifold $X$ admits a proper holomorphic embedding into $\C^m$ for some $m$.  In 1960--61, E.\ Bishop \cite{Bishop1961} and R.\ Narasimhan \cite{Narasimhan1960} independently showed that if $\dim_\C X=n$, then $m$ can be taken to be $2n+1$.  The optimal embedding result is that if $n\geq 2$, then $m$ can be taken to be $\left[3n/2\right]+1$.  This was conjectured by O.\ Forster in 1971 \cite{Forster1970}; he showed that for each $n$, no smaller value of $m$ works in general.  Forster's conjecture was proved in the early 1990s by Y.\ Eliashberg and M.\ Gromov \cite{Eliashberg-Gromov2} (following their much earlier paper \cite{Eliashberg-Gromov1}) and J.\ Sch\"urmann \cite{Schurmann}.  The proof relies on Gromov's Oka principle discussed in \S \ref{ss:ellipticsub} above. This problem is still open in dimension one: every open Riemann surface properly embeds into $\C^3$, but relatively few are known to embed, even non-properly, into $\C^2$.  For recent results in this direction, see \cite{Forstneric-Wold1,Forstneric-Wold3, Majcen-2009}.

Third, Stein manifolds are characterized by a cohomology vanishing property.  The famous Theorem B of Henri Cartan, proved in his seminar in the period 1951--4, states that $H^k(X;\mathscr F)=0$ for every coherent analytic sheaf $\mathscr F$ on a Stein manifold (or Stein space) $X$ and every $k\geq 1$.  The converse is easy.

Finally, Stein manifolds can be defined in terms of {\em plurisubharmonicity}, a notion introduced independently by P.\ Lelong and by K.\ Oka in 1942, which plays a fundamental role in complex analysis. Recall that the {\em Levi form} $\cL_\rho$ of a $\cC^2$ function $\rho$ on a complex manifold $X$ is a quadratic Hermitian form on the tangent bundle $TX$ that is given in any local holomorphic coordinate system $z=(z_1,\ldots,z_n)$ on $X$ by the {\em complex Hessian}:
\[
   \cL_\rho(z;v)=\sum_{j,k=1}^n \frac{\di^2 \rho(z)}{\di z_j\di\bar z_k}\, v_j\ol v_k,\quad v\in \C^n.
\]
The function $\rho$ is plurisubharmonic if $\cL_\rho\ge 0$, and is strongly plurisubharmonic if $\cL_\rho>0$; clearly the latter is an open condition. For an intrinsic definition of the Levi form, associated to the $(1,1)$-form $dd^c\rho=\I\, \di\dibar\rho$, see e.g.\ \cite[Sec.\ 1.8]{FF:book}.

Every Stein manifold $X$ admits many strongly plurisubharmonic functions, for example, those of the form $\rho=\sum_j |f_j|^2$, where the holomorphic functions $f_j\in \cO(X)$ are chosen so that their differentials span the cotangent space $T_x^* X$ at each point $x\in X$ and the series converges. It is easy to ensure in addition that $\rho$ is an {\em exhaustion function}, in the sense that for every $c\in \R$, the sublevel set $\{x\in X \colon \rho(x) \le c\}$ is compact.  Conversely, H.~Grauert proved in 1958 \cite{Grauert:Levi} that a complex manifold $X$ with a strongly plurisubharmonic exhaustion function $\rho:X\to\R$ is Stein; this is the solution of the Levi problem for manifolds. The main point is that the existence of a strongly plurisubharmonic exhaustion $\rho$ on a complex manifold $X$ implies the solvability of all consistent $\overline\partial$-equations on $X$. In particular, the set $\rho^{-1}(-\infty,c_1)$ is Runge in $\rho^{-1}(-\infty,c_2)$ for all real numbers $c_1<c_2$, from which the defining properties (1) and (2) of a Stein manifold easily follow. A quantitative treatment of the $\dibar$-problem, using Hilbert space methods, was given by L.\ H\"ormander in 1965 \cite{Hormander:ACTA,Hormander:SCV}.

In Oka theory we need to use a sup-norm bounded solution operator to the $\dibar$-equation on bounded strongly pseudoconvex Stein domains. Such a domain $D\Subset X$ in a complex manifold  $X$ is of the form $D=\{x\in U\colon \rho(x)<0\}$, where $\rho$ is a strongly plurisubharmonic function on an open set $U\supset \bar D$ and $d\rho(x)\ne 0$ for each $x\in bD=\{\rho=0\}$. A key geometric property of a strongly pseudoconvex domain is that it is locally near each boundary point biholomorphic to a convex domain in $\C^n$. 

Denote by $\cC^r_{0,1}(\bar D)$ the space of all $(0,1)$-forms of class $\cC^r$ on $\bar D$; for $r=1/2$ this is the H\"older class with H\"older exponent $1/2$. An operator $T$ in the following theorem is obtained as an integral kernel operator with holomorphic kernel; we refer to Henkin and Leiterer \cite{Henkin-Leiterer:TF} and Range and Siu \cite{Range-Siu1973}.

\begin{theorem}
\label{th:dibar-spsc}
For every bounded strongly pseudoconvex Stein domain $D$ with $\cC^2$ boundary in a complex manifold $X$ there exists a linear operator $T\colon \cC^0_{0,1}(D)\to \cC^{1/2}(D)$ such that, if $f\in \cC^0_{0,1}(\bar D)\cap \cC^1_{0,1}(D)$ and $\dibar f=0$ in $D$, then 
\[
    \dibar (Tf) = f,\quad     ||Tf||_{\cC^{1/2}(\bar D)} \le c_D\, ||f||_{\cC^{0}_{0,1}(\bar D)}.
\]
The constant $c_D$ can be chosen uniform for all domains sufficiently $\cC^2$-close to $D$. 
\end{theorem}

The definition of a complex space being Stein is the same as for manifolds.  Theorem B of Cartan still holds.  The property of having a strongly plurisubharmonic exhaustion function is still equivalent to being Stein (Narasimhan 1962, \cite{Narasimhan1962}). Finally, a Stein complex space with a bound on the local embedding dimension is biholomorphic to a complex subspace of a Euclidean space (Narasimhan 1960, \cite{Narasimhan1960}). 
Often the most efficient way to show that a complex space is Stein is to find a strongly plurisubharmonic exhaustion on it.  This is how Y.-T.\ Siu proved in 1976 that a Stein subvariety of a reduced complex space has a basis of open Stein neighborhoods \cite{Siu1976}.  Stein neighbourhood constructions often allow us to transfer a problem on a complex space to a Euclidean space where it becomes tractable.  A recent example is the application of the Stein neighbourhood construction in \cite{Forstneric-Wold2} to the proof in \cite{FF:OkaManifolds} that the basic Oka property implies the parametric Oka property for manifolds; see \cite[Chap.\ 3]{FF:book} for more on this subject.

The Morse index of a nondegenerate critical point of a strongly plurisubharmonic function on a complex manifold $X$ is at most $n=\dim X$.  In fact, the quadratic normal form of such a function in local coordinates is given by Lemma \ref{spsh-normal-form} below. By Morse theory, it follows that a Stein manifold $X$ of complex dimension $n$ has the homotopy type of a CW complex of real dimension at most $n$. This observation has a highly nontrivial converse, Eliashberg's topological characterization of Stein manifolds of dimension at least $3$ \cite{Eliashberg1}:  If $(X,J)$ is an almost complex manifold of complex dimension $n\geq 3$, which admits a Morse exhaustion function $\rho:X\to\R$ all of whose Morse indices are at most $n$, then $J$ is homotopic to an integrable Stein structure $\wt J$ on $X$ in which the sublevel sets of $\rho$ are strongly pseudoconvex.  

Eliashberg's result fails in dimension $2$.  The simplest counterexample is the smooth 4-manifold $S^2\times\R^2$, which satisfies the hypotheses of the theorem but carries no Stein structure by Seiberg-Witten theory.  Namely, $S^2\times\R^2$ contains embedded homologically nontrivial spheres $S^2\times\{c\}$, $c\in\R^2$, with self-intersection number $0$, while the adjunction inequality shows that any homologically nontrivial smoothly embedded sphere $C$ in a Stein surface has $C\cdot C\leq -2$ \cite[\S 9.8]{FF:book}. Nevertheless, $S^2\times \R^2$ still admits {\em exotic Stein structures}. More precisely, R.\ Gompf has shown that a topological oriented 4-manifold is homeomorphic to a Stein surface if (and only if) it is the interior of a topological handlebody with only $0$-, $1$-, and $2$-handles  \cite{Gompf1,Gompf2}.  In particular, there are Stein surfaces homeomorphic to $S^2\times\R^2$. See \cite[\S 9.12]{FF:book} for an exposition.

The Eliashberg-Gompf theorem has the following extension, due to Forstneri\v c and Slapar \cite[Theorem 1.1]{FSlapar1}, which shows that the Oka principle holds for maps to an arbitrary complex manifold, provided that we are allowed to deform not only the map, but also the Stein structure on the source manifold. 

%
%
%  FS1: MAIN THEOREM
%
%
\begin{theorem}
\label{FS1:Main1}
{\rm (The Soft Oka Principle)} 
Let $(X,J)$ be an almost complex manifold of dimension $n$ which admits a Morse exhaustion function $\rho\colon X\to\R$ with all Morse indices $\le n$. Let $f\colon X\to Y$ be a continuous map to a complex manifold~$Y$. 
\begin{itemize}
\item[\rm (i)] If \ $\dim_\C X \ne 2$, then there exist an integrable Stein structure $\wt J$ on $X$, homotopic to $J$, and 
a $\wt J$-holomorphic map $\wt f\colon X\to Y$ homotopic to~$f$.
\item[\rm (ii)] 
If $\dim_\C X =2$, there exist an orientation preserving homeomorphism $h\colon X\to X'$ onto a Stein surface $X'$ and a holomorphic map 
$f'\colon X' \to Y$ such that the map $\wt f= f'\circ h \colon X\to Y$ is homotopic to $f$.
\end{itemize}
Furthermore, a family of maps $f_p\colon X\to Y$, depending continuously on the parameter $p$ in a compact Hausdorff space, can be  deformed to a family of holomorphic maps with respect to some Stein structure $\wt J$ on $X$ that is homotopic to $J$.    
\end{theorem}

%%%%%%%%%%%%%%%%%%%%%%%%%%%%%%%%%%%%%%%%%%%%%%%%%%%%%%%%%%%%%%%%%
%																																%
%   CARTAN PAIRS AND CONVEX BUMPS        												%
%																																%
%%%%%%%%%%%%%%%%%%%%%%%%%%%%%%%%%%%%%%%%%%%%%%%%%%%%%%%%%%%%%%%%%	
\subsection{Cartan pairs and convex bumps}
\label{ss:Cartanpair}
The main references for this subsection are sections 5.7 and 5.10 in \cite{FF:book}. Recall that a compact set $K$ in a complex space $X$ is a {\em Stein compactum} if $K$ admits a basis of open Stein neighborhoods in $X$. The following definition combines Def.\ 5.7.1 and 5.10.2 in \cite{FF:book}, suitably simplified for our present needs. 
 
%
%
%  Definition of a Cartan pair
%
%
\begin{definition}
\label{Cartan-pair}
{\rm (I)}
A pair $(A,B)$ of compact subsets in a complex space $X$ is said to be a {\em Cartan pair} if it satisfies the following two conditions:    
\begin{itemize}
\item[\rm(i)]  the sets $A$, $B$, $D=A\cup B$ and $C=A\cap B$ are Stein compacta, and
\item[\rm(ii)]  $A,B$ are {\em separated}, in the sense that $\overline{A\bs B}\cap \overline{B\bs A} =\varnothing$.
\end{itemize}
{\rm (II)}
A pair $(D_0,D_1)$ of open relatively compact subsets of a complex manifold $X$ is said to be a {\em \SPCP\ of class $\cC^\ell$} $(\ell\ge 2)$ if $(\bar D_0,\bar D_1)$ is a Cartan pair in $X$, and each of the sets $D_0$, $D_1$, $D=D_0\cup D_1$, and $D_{0,1}=D_0\cap D_1$ is a \spsc\ Stein domain with $\cC^\ell$ boundary.
\item{\rm (III)}
A Cartan pair $(A,B)$ in a complex space $X$ is said to be a {\em special Cartan pair}, and the set $B$ is a {\em convex bump on $A$}, if $B\subset X_{\reg}$ and there exist holomorphic coordinates in a neighborhood $U\subset X_{\reg}$ of $B$ in which the sets $B$, $A\cap B$, and $U\cap (A\cup B)$ are strongly convex with $\cC^2$ boundaries in $U$. (See Fig.\ 5.2 in \cite[p.\ 219]{FF:book}.)
\end{definition}

Every Cartan pair in a complex manifold can be approximated from the outside by smooth strongly pseudoconvex Cartan pairs as in the following proposition.

%
%
%  SMOOTHING A CARTAN PAIR
%
%
\begin{proposition}
\label{SmoothingCP}
{\rm \cite[Proposition 5.7.3, p.\ 210]{FF:book}} Let $(A,B)$ be a Cartan pair in a complex manifold $X$. Given open sets $U\supset A$, $V\supset B$ in $X$, there exists a \SPCP\ $(D_0,D_1)$ of class $\cC^\infty$ satisfying $A\subset D_0\Subset U$ and $B\subset D_1\Subset V$.
\end{proposition}

We also need to recall the following notion.

\begin{definition} 
\label{def:noncritical-extension}
Assume that $A\subset A'$ are compact strongly pseudoconvex domains in a complex manifold $X$. We say that $A'$ is a {\em noncritical \spsc\ extension} of $A$ if there exist a \spsh\ function $\rho$ in an open set $V\supset \overline{A'\bs A}$ in $X$, with $d\rho\ne 0$ on $V$, and real numbers $c<c'$ such that 
\[
    A\cap V =\{x\in V\colon \rho(x)\le c\}, \quad    A'\cap V=\{x\in V\colon \rho(x)\le c'\}.
\]
\end{definition}

The sets $A_t:=A\cup \{x\in V\colon \rho(x)\le t\}$ for $t\in [c,c']$ are an increasing family of strongly pseudoconvex domains with $A_c=A$ and $A_{c'}=A'$. If $A$ is Stein, then $\rho$ extends to $A$ as a strongly plurisubharmonic function, and all domains $A_t$ for $t\in [c,c']$ are Stein.

The following result is a simple consequence of the fact that a strongly pseudoconvex hypersurface is locally near each point strongly convex in an appropriately chosen system of local holomorphic coordinates. Similar results have been used before, in particular by Grauert and in the Andreotti-Grauert theory.

\begin{lemma}
\label{lem:extension-by-bumps}
{\rm \cite[Lemma 5.10.3, p.\ 218]{FF:book}}
Assume that $A\subset A'$ are compact strongly pseudoconvex domains in a complex manifold $X$ and $A'$ is a noncritical \spsc\ extension of $A$. There exists a finite sequence of compact \spsc\ domains $A=A_0\subset A_1 \subset \cdots \subset A_m=A'$ such that for every $k=0,1,\ldots,m-1$ we have $A_{k+1}=A_k\cup B_k$, where $B_k$ is a convex bump on $A_k$. In addition, given an open cover $\cU=\{U_j\}$ of $\overline{A'\bs A}$, we can choose the above sequence such that each $B_k$ is contained in some $U_{j(k)}$.  
\end{lemma}

%%%%%%%%%%%%%%%%%%%%%%%%%%%%%%%%%%%%%%%%%%%%%%%%%%%%%%%%%%%%%%%%%
%																																%
%   A SPLITTING LEMMA                   												%
%																																%
%%%%%%%%%%%%%%%%%%%%%%%%%%%%%%%%%%%%%%%%%%%%%%%%%%%%%%%%%%%%%%%%%	
\subsection{A splitting lemma}
\label{ss:splitting}
The classical {\em Cartan lemma} (see e.g.\ \cite[p.\ 88]{Grauert-Remmert}) states that, given a special configuration of compact sets $(A,B)$ in $\C^n$ (for example, a pair of cubes whose union and intersection are also cubes), every holomorphic map $\gamma\colon A\cap B\to G$ to a complex Lie group $G$ splits as a product $\gamma=\beta \cdot \alpha^{-1}$ of holomorphic maps $\alpha\colon A\to G$, $\beta\colon B \to G$. Cartan considered the case $G=GL_k(\C)$, but the same proof applies in general. The original application of Cartan's lemma was to glue a pair of syzygies (resolutions of a coherent analytic sheaf) over $A$ and $B$, thereby obtaining a syzygy over $A\cup B$. This is one if the main steps in the proof of Cartan's Theorems A and B for coherent analytic sheaves on Stein spaces. 

Later on, Grauert used Cartan's lemma to develop his Oka-Grauert theory. If a holomorphic principal fiber bundle over $A\cup B$ is trivial over both $A$ and $B$ individually, then by splitting the transition map over $A\cap B$ by using Cartan's lemma, we can precompose the trivialization maps over $A$ and $B$ by $\alpha$ and $\beta$, respectively. In this way we obtain new trivializations that match over $A\cap B$, so they amalgamate into a trivialization of the bundle over the union $A \cup  B$.

Cartan's lemma no longer suffices in the nonlinear setting of general Oka theory, in the absence of any special structure on the fibers. Instead we use a method of {\em gluing local holomorphic sprays}, which may for obvious reasons also be called {\em thick holomorphic sections}. In this section we present the relevant splitting lemma. It was first used in \cite{FF:CAP} and was later improved in \cite{BDF1,FF:manifolds}; the version presented here is taken from \cite[\S 5.8]{FF:book}. The main point is that the product splitting in Cartan's lemma is replaced by the {\em compositional splitting} in (\ref{eq:splitting}) below.

Given a compact subset $K$ in a complex manifold $X$ and an open set $W\subset \C^N$, we consider maps $\gamma \colon K\times W\to K\times \C^N$ of the form
\begin{equation}
\label{map-gamma}
     \gamma(x,w)=\bigl(x,\psi(x,w)\bigl),\quad x\in K, \ w\in W.
\end{equation}
We say that $\gamma \in \cA(K\times W)$ if $\gamma$ is continuous on $K\times W$ and holomorphic in $\mathring K \times W$. Let $\Id$ denote the identity map;  then 
\[
        \dist_{K \times W}(\gamma,\Id) =  \sup\{ |\psi(x,w)-w|\colon x\in K, \ w\in W\}.
\]

%
%
%  Splitting lemma for open Cartain pair of class C^2.
%
%

\begin{proposition}
\label{splitting}
{\rm \cite[Proposition 5.8.1, p.\ 211]{FF:book}}
Let $(D_0,D_1)$ be a \SPCP\ of class $\cC^2$ in a complex manifold $X$ (Def.\ \ref{Cartan-pair}). Set $D_{0,1}=D_0\cap D_1$ and $D=D_0\cup D_1$.
Given a bounded open convex set $0\in W\subset\C^N$ and a number $r\in(0,1)$, there is a $\delta>0$ satisfying the following. For every map $\gamma \colon \bar D_{0,1}\times W \to \bar D_{0,1}\times \C^N$ of the form (\ref{map-gamma}) and of class $\cA(D_{0,1}\times W)$, with $\dist_{\bar D_{0,1} \times W}(\gamma,\Id) <\delta$, there exist maps
\[
  \alpha_\gamma \colon \bar D_0 \times r W \to \bar D_0 \times \C^N, \quad
  \beta_\gamma \colon \bar D_1 \times r W \to \bar D_1 \times \C^N
\]
of the form (\ref{map-gamma}) and of class $\cA(D_0\times rW)$ and $\cA(D_1\times rW)$, respectively, depending smoothly on $\gamma$, such that $\alpha_{\Id}=\Id$, $\beta_{\Id}=\Id$, and
\begin{equation}
\label{eq:splitting}
   \gamma\circ\alpha_\gamma = \beta_\gamma    \quad {\rm on\ } \bar D_{0,1}\times rW.
\end{equation}
\end{proposition}

The proof gives a number of additions at no extra cost. For example, if $\gamma$ agrees with the identity to order $m\in\N$ along $w=0$, then so do $\alpha_\gamma$ and~$\beta_\gamma$. Furthermore, if $X'$ is a closed complex subvariety of $X$ such that $X'\cap \bar D_{0,1}=\varnothing$, then $\alpha_\gamma$ and $\beta_\gamma$ can be chosen tangent to the identity to any given finite order along $(X' \cap \bar D_0)\times rW$ and $(X' \cap \bar D_1)\times rW$, respectively.

\begin{proof}
Denote by $C_r$ and $\Gamma_r$ the Banach spaces of all continuous maps $\bar D_{0,1} \times rW \to \C^N$ which are holomorphic in $D_{0,1} \times rW$ and satisfy
\begin{eqnarray*}
			||\phi||_{C_r}     &=& \sup\{ |\phi(x,w)| \colon x \in \bar D_{0,1}, \ w\in rW \}  < +\infty, \\
     ||\phi||_{\Gamma_r} &=& \sup\{  \bigl( |\phi(x,w)| + |\di_w \phi(x,w)| \bigr) \colon 
     x \in \bar D_{0,1}, \ w\in rW \} < +\infty.
\end{eqnarray*}
%
%Here, $\di_w$ denotes the partial differential with respect to the variable $w\in\C^N$, and $|\di_w \phi(x,w)|$ is the Euclidean operator norm.
%
Similarly we denote by $A_{r}$, $B_{r}$ the Banach space of all continuous maps $\bar D_0\times r W \to \C^N$, $\bar D_1\times r W \to \C^N$, respectively, 
that are holomorphic in the interior. 

By the hypothesis we have $\gamma(x,w)=(x,\psi(x,w))$ with $\psi\in C_{1}$. Set $\psi_0(x,w)=w$. Choose a number $r_1 \in (r,1)$. By the Cauchy estimates, the restriction map $C_{1} \to \Gamma_{r_1}$ is continuous;  hence $\psi|_{\bar D_{0,1} \times r_1 W} \in \Gamma_{r_1}$ and
$||\psi-\psi_0||_{\Gamma_{r_1}} \le \mathrm{const} ||\psi-\psi_0||_{C_1}$.

%
%
%    LEMMA ON LINEAR SPLITTING
%
%
\begin{lemma}
\label{linear-splitting}
There are bounded linear operators $\Agoth \colon C_{r} \to A_{r}$, $\Bgoth \colon C_{r}\to B_{r}$, satisfying
\begin{equation}
\label{left-right}
        c = \Agoth c - \Bgoth c,\qquad c\in C_{r}.
\end{equation}
If $c$ vanishes to order $m\in\N$ at $w=0$ then so do $\Agoth c$ and $\Bgoth c$. If $X'$ is a closed complex subvariety of $X$ such that $X'\cap \bar D_{0,1}=\varnothing$, then $\Agoth c$ and $\Bgoth c$ can be chosen to vanish to any given finite order along $(X' \cap \bar D_0)\times rW$ and $(X' \cap \bar D_1)\times rW$, respectively.
\end{lemma}

\begin{proof}
By condition (ii) in Def.\ \ref{Cartan-pair} there is a smooth function $\chi \colon X \to [0,1]$ such that $\chi=0$ in a neighborhood of $\overline {D_0\backslash D_1}$, and $\chi=1$ in a neighborhood of $\overline{D_1\bs D_0}$. For any $c\in C_r$, the product $\chi(x) c(x,w)$ extends to a continuous function on $\bar D_0\times  W$ that vanishes on $\overline {D_0\bs D_1} \times  W$, and $(\chi(x)-1)c(x,w)$ extends to a continuous function on
$\bar D_1\times  W$ that vanishes on $\overline {D_1\bs D_0} \times  W$. Furthermore, 
$ %\[
	\dibar(\chi c)= \dibar((\chi-1)c)=c\dibar \chi 
$ %\]
is a $(0,1)$-form on $\bar D$ with continuous coefficients, with support in $\bar D_{0,1} \times W$ and depending holomorphically on $w\in W$.

Choose functions $f_1,\ldots, f_m\in \cO(X)$ that vanish to order $k$ along the subvariety $X'$ and have no common zeros on $\bar D_{0,1}$. Since $D_{0,1}$ is strongly pseudoconvex, Cartan's division theorem gives holomorphic functions $g_1,\ldots,g_m$ in a neighborhood of $\bar D_{0,1}$ such that $\sum_{j=1}^m f_j g_j=1$. Let $T\colon \cC^0_{0,1}(D)\to \cC^0(D)$ be a bounded linear solution operator furnished by Theorem \ref{th:dibar-spsc}. For any $c \in C_r$ and $w\in rW$ we set 
\begin{eqnarray*}
    (\Agoth c)(x,w) &=& \chi(x)\, c(x,w) - 
    \sum_{j=1}^m f_j(x)\, T\bigl( g_j\, c(\cdotp,w) \, \dibar \chi\bigr) (x),
        \qquad\ \quad   x\in\bar D_0, \\
    (\Bgoth c)(x,w) &=& \bigl(\chi(x)-1\bigr) c(x,w)  
    		-  \sum_{j=1}^m f_j(x) \, T\bigl(g_j\, c(\cdotp,w)\, \dibar \chi\bigr) (x),
        		\quad x\in\bar D_1.
\end{eqnarray*}
Then $\Agoth c-\Bgoth c=c$ on $\bar D_{0,1}\times rW$, and $\dibar_x (\Agoth c)=0$, $\dibar_x(\Bgoth c)=0$ in the interior of their respective domains. 
Since $\dibar_w \left( c(x,w) \,\dibar \chi(x)\right)=0$ and $T$ commutes with $\dibar_w$, we also have $\dibar_w (\Agoth c)=0$ and $\dibar_w (\Bgoth c)=0$.
The estimates follow from boundedness of $T$. 
\end{proof}

Consider the following operator, defined for $\psi\in \Gamma_{r_1}$ near $\psi_0$ and $c\in C_{r}$ near $0$:
\[
        \Phi(\psi,c)(x,w) = \psi\bigl(x,w+(\Agoth c)(x,w)\bigr) - \bigl(w+(\Bgoth c)(x,w)\bigr),\quad 
        x\in \bar D_{0,1},\ w\in rW.
\]
Then $(\psi,c)\mapsto \Phi(\psi,c)$ is a smooth map from an open neighborhood of $(\psi_0,0)$ in the Banach space $\Gamma_{r_1} \times C_{r}$ to the Banach space $C_{r}$. Indeed, $\Phi$ is linear in $\psi$, and we have
\[
     \di_c\, \Phi(\psi,c_0)c(x,w) =
     \di_c\, \psi\bigl(x, w+(\Agoth c_0)(x,w)\bigr) \cdotp (\Agoth c)(x,w) - (\Bgoth c)(x,w);
\]
this is again linear in $\psi$ and continuous in all variables. A similar argument applies to the higher order differentials of $\Phi$. By (\ref{left-right}) we have
\[
	\Phi(\psi_0,c)= \Agoth(c)-\Bgoth(c)=c,\quad c\in C_r,
\]
so $\di_c\, \Phi(\psi_0,0)$ is the identity map on $C_r$. By the implicit function theorem there exists a smooth map $\psi \mapsto \Cgoth(\psi) \in C_{r}$ in an open neighborhood of $\psi_0$ in $ \Gamma_{r_1}$ such that $\Phi(\psi,\Cgoth(\psi))=0$ and $\Cgoth(\psi_0)=0$. The maps
\[
        a_\psi(x,w) = w + \Agoth\circ \Cgoth(\psi)(x,w), \quad  b_\psi(x,w) = w + \Bgoth\circ \Cgoth(\psi)(x,w)
\]
then satisfy $a_\psi\in A_{r}$,  $a_{\psi_0}=\psi_0$, $b_\psi\in B_r$, $b_{\psi_0}=\psi_0$ and
\[
        \psi\bigl(x,a_\psi(x,w)\bigr) = b_\psi(x,w),\quad (x,w)\in \bar D_{0,1} \times rW.
\]
The associated holomorphic maps 
\[
	\alpha_\gamma (x,w) = \bigl(x,a_\psi(x,w)\bigr), \quad \beta_\gamma (x,w)  = \bigl(x,b_\psi(x,w)\bigr)
\]
 depend smoothly on $\gamma$ and satisfy Proposition \ref{splitting}.
\end{proof}

%%%%%%%%%%%%%%%%%%%%%%%%%%%%%%%%%%%%%%%%%%%%%%%%%%%%%%%%%%%%%%%%%
%																																%
%   GLUING SPRAYS OF HOLOMORPHIC SECTIONS												%
%																																%
%%%%%%%%%%%%%%%%%%%%%%%%%%%%%%%%%%%%%%%%%%%%%%%%%%%%%%%%%%%%%%%%%	

\subsection{Gluing holomorphic sprays of sections}
\label{ss:sprays}
We now apply Proposition \ref{splitting} to glue local holomorphic sprays over a Cartan pair. Our main reference is \cite[\S 5.9]{FF:book}.

Assume that $h \colon Z\to X$ is a holomorphic submersion of complex spaces. For each point $z\in Z$ we let $VT_z Z= \Ker \d h_z$ denote the vertical tangent space of $Z$ (i.e., the tangent space to the fiber $h^{-1}(h(z))$) at $z$.

%
%  SPRAYS
%
%
\begin{definition}
\label{def:local-spray}
{\rm \cite[Def.\ 5.9.1, p.\ 215]{FF:book}}
Let $D$ be a domain in $X$. A {\em holomorphic spray of sections} over $D$ is a holomorphic map $f\colon D\times P\to Z$, where $P$ is an open set in some Euclidean space $\C^N$ containing the origin, such that $h(f(x,w))=x$ for every $x \in D$ and $w\in P$. The spray is said to be {\em dominating} on a set $K\subset D$ if the partial differential
\begin{equation}
\label{pispray}
    \di_w|_{w=0} f(x,w) \colon \T_0 \C^N \cong \C^N  \longrightarrow  VT_{f(x,0)} Z
\end{equation}
is surjective for all $x\in K$; $f$ is dominating if this holds for $K=D$. We call $f_0=f(\cdotp,0)$ the {\em central (core) section}  of the spray $f$. 

If $D \Subset X_\reg$ is a relatively compact domain with $\cC^1$ boundary then a {\em spray of sections of class $\cA^l(D)$} (with the parameter set $P\subset \C^N$) is a $\cC^l$ map $f\colon \bar D\times P\to Z$ that is holomorphic on $D\times P$ and satisfies the condition (\ref{pispray}).
\end{definition}

The following lemma shows that every holomorphic section over a Stein domain can be embedded into a dominating spray of sections on a slightly smaller domain.

%
%
%  EXISTENCE OF LOCAL SPRAYS
%
%
\begin{lemma}
\label{local-spray}
{\rm \cite[Lemma 5.10.4]{FF:book}}
Assume that $h \colon Z\to X$ is a holomorphic submersion onto a complex space $X$, $X'$ is a closed complex subvariety of $X$, and $A_0 \subset A$ are Stein compacta in $X$ such that $A_0\cap  X'=\varnothing$. Given open Stein sets $V\Subset V_0\subset X$ containing $A$, a holomorphic section $f\colon V_0\to Z$ and an integer $r \in\N$,
there is a holomorphic spray of sections $F\colon V\times P \to Z$ for some open set $0\in P\subset \C^N$ such that
\begin{itemize}
\item[\rm (i)]  $F_w=F(\cdotp,w)$ is a section of $Z|_V$ for every fixed $w\in P$, with $F_0=f$,
\item[\rm (ii)]  $F_w$ agrees with $f$ to order $r$ along $V\cap X'$ for every $w\in P$, and
\item[\rm (iii)] $\di_w F(x,w)|_{w=0}\colon \C^N\to VT_{f(x)} Z$ is surjective for every $x\in A_0$.
\end{itemize}
If $A_0$ is $\cO(A)$-convex and it admits a contractible Stein neighborhood in $V_0 \bs X'$ then the above conclusion holds with $N =\dim h^{-1}(x)$ $(x\in A_0)$.
\end{lemma}

\begin{proof}
The image $f(V_0)$ is a closed Stein subvariety of $Z|_{V_0}$, and hence it admits an open Stein neighborhood $\Omega\subset Z$. By Cartan's Theorem A there exist finitely many holomorphic vector fields $v_1,\ldots,v_N$ on $\Omega$ that are tangent to the fibers of $h$, that span the vertical tangent bundle $VT Z$ at every point of $f(A_0)$, and that vanish to order $r$ on the subvariety $h^{-1}(X') \cap\Omega$. Let $\theta^t_j$ denote the flow of $v_j$. The map
\[
    F(x,w_1,\ldots,w_N)= \theta^{w_1}_1\circ\cdots\circ \theta^{w_N}_N \circ f(x)
\]
is defined and holomorphic for all $x\in V$ and for all $w=(w_1,\ldots,w_N)$ in an open neighborhood of the origin in $\C^N$. Since $\di_{w_j} F(x,w)|_{w=0} = v_j(f(x))$ and the vector fields $v_j$ span $VT Z$ along $f(A_0)$, the differential $\di_w F(x,w)|_{w=0}  \colon \C^N \to VT_{f(x)} Z$ is surjective at every point $x \in A_0$.

For the proof of the last claim (which will not be important here) see \cite[p.\ 220]{FF:book}.
\end{proof}

\begin{remark}
Similarly it can be shown that, given a holomorphic submersion $Z\to X$, a strongly pseudoconvex Stein domain $D\Subset X$ and a section $f\colon \bar D\to Z$ of class $\cA^r(D)$ for some $r\in \Z_+$, there exists a dominating spray $F\colon \bar D\times P\to Z$ of class $\cA^r(D)$ with $F(\cdotp,0)=f$ (see \cite{BDF1,FF:manifolds}). This is used in the proof of an up-to-the-boundary version of Theorem \ref{th:CAP} \cite{BDF2}, in the construction of open Stein neighborhoods of the image $f(\bar D)\subset Z$ \cite[Theorem 1.2]{FF:manifolds}, and in many other applications.
\end{remark}

The following result is our main {\em  gluing lemma} for holomorphic sprays. This result can be viewed as a solution of a nonlinear Cousin-I problem.

%
%
%  THE MAIN HEFTUNGSLEMMA
%
%
\begin{proposition}
\label{gluing-sprays}
{\rm \cite[Proposition 5.9.2, p.\ 216]{FF:book}}
Assume that $h\colon Z\to X$ is a holomorphic submersion and $(D_0,D_1)$ is a \SPCP\ of class $\cC^\ell$ $(\ell \ge 2$) in $X$ (Def.\ \ref{Cartan-pair}). Set $D=D_0\cup D_1$ and $D_{0,1}=D_0\cap D_1$. Let $l\in \{0,1,\ldots,\ell\}$.

Given a holomorphic spray of sections $f\colon \bar D_0\times P_0\to X$ $(P_0\subset\C^N)$
of class $\cA^l(D_0)$ which is dominating on $\bar D_{0,1}$,
there is an open set $P\subset\C^N$ with $0\in P \subset P_0$
satisfying the following property. For every holomorphic spray of sections
$g \colon \bar D_1 \times P_0 \to X$ of class $\cA^l(D_1)$
which is sufficiently close to $f$ in $\cC^l(\bar D_{0,1} \times P_0)$
there exists a holomorphic spray of sections $f' \colon \bar D \times P \to X$
of class $\cA^l(D)$, close to $f$ in $\cC^l(\bar D_0\times P)$
(depending on the $\cC^l$-distance between $f$ and $g$ on
$\bar D_{0,1} \times P_0$), whose core section $f'_0$ is homotopic to $f_0$ on $\bar D_0$
and is homotopic to $g_0$ on $\bar D_1$.

If $f$ and $g$ agree to order $m\in\Z_+$ along $\bar D_{0,1}\times\{0\}$, then $f'$
can be chosen to agree to order $m$ with $f$ along $\bar D_{0}\times\{0\}$, and with
$g$ along $\bar D_{1}\times\{0\}$.

If $\sigma$ is the zero set of finitely many
$\cA^l(D_0)$ functions and $\sigma\cap\bar D_{0,1}=\emptyset$,
then $f'$ can be chosen so that $f'_0$
agrees with $f_0$ to a finite order on~$\sigma$.
\end{proposition}

\begin{proof}
By Lemma 5.9.3 in \cite[p. 216]{FF:book} there exist a domain $P_1\Subset P_0$ containing the origin and a transition map $\gamma \colon\bar D_{0,1} \times P_1 \to \bar D_{0,1}\times \C^N$ of the form 
\[
	\gamma(x,w)=(x,\psi(x,w)),\qquad x\in \bar D_{0,1},\ w\in P_1
\]
and of class $\cA^l(D_{0,1}\times P_1)$, close to the identity map $\Id(x,w)= (x,w)$ in the $\cC^l$ topology (the closeness depending on the $\cC^l$ distance between $f$ and $g$ on $\bar D_{0,1} \times P_0$), satisfying
\[
    f=g \circ\gamma \qquad {\rm on}\ \bar D_{0,1}\times P_1.
\]
(The cited result follows from the implicit function theorem and the fact that every complex vector subbundle of class $\cA^l(D_{0,1})$  of the trivial bundle $\bar D_{0,1}\times \C^N$ admits a complementary subbundle of the same class.)

Let $P \Subset P_1$ be a domain containing the origin $0\in\C^N$. If $\gamma$ is sufficiently $\cC^l$-close to the identity on $\bar D_{0,1}\times P_1$ then by Proposition \ref{splitting} there exist maps
\[
    \alpha \colon\bar D_0\times P\to \bar D_0\times\C^N,\quad   \beta\colon\bar D_1\times P\to \bar D_1\times\C^N,
\]
of class $\cA^l$ on their respective domains and satisfying
\[
    \gamma\circ\alpha = \beta \qquad {\rm on}\ \bar D_{0,1}\times P.
\]
From this and $f=g \circ\gamma$ it follows that
\[
    f\circ\alpha = g\circ \beta \quad \hbox{on}\ \bar D_{0,1}\times P.
\]
Hence $f\circ\alpha$ and  $g\circ \beta$ amalgamate into a holomorphic spray $f'\colon \bar D\times P\to Z$ with the stated properties.

The additions concerning interpolation are easily achieved by using the corresponding version of Proposition \ref{splitting}; see the paragraph following it.  
\end{proof}

Using these tools we now prove the main building block for our Theorem \ref{th:CAP}.

\begin{lemma}[\bf Approximate extension to a special convex bump]
\label{extending-to-bump}
Assume that $X$ is a complex space, $X'\subset X$ is a closed complex subvariety containing $X_\sing$, and $(A,B)$ is a special Cartan pair in $X$. Let $\pi \colon Z\to X$ be a holomorphic submersion. Assume that there is an open set $U\subset X_\reg$ with $B\subset U$ such that the restriction $Z|_{U}$ is isomorphic to a trivial bundle $U\times Y$ whose fiber $Y$ enjoys \CAP. Then every holomorphic section $f\colon V_0\to Z$ on an open set $V_0\supset A$ can be approximated uniformly on $A$ by sections $f'$ that are holomorphic over an open neighborhood of $A\cup B$ and agree with $f$ to a given order $m\in\N$ along the subvariety $X'$. 
\end{lemma}

\begin{proof} We proceed in three steps.

{\em Step 1: Thickening.} Lemma \ref{local-spray} furnishes an open set $V$ in $X$ with $A\subset V\subset V_0$, an open set $0\in W\subset \C^N$,
and a holomorphic spray of sections $F\colon V\times W\to Z$ with the core $f=F(\cdotp,0)$ on $V$. Using the trivialization $Z|_U\cong U\times Y$ we can write $F(x,w)=(x,F'(x,w))\in X\times Y$ for $x\in U\cap V$ and $w\in W$. Thus $F'$ is a holomorphic map $(U\cap V) \times W\to Y$ with values in $Y$.

\smallskip {\em Step 2: Approximation.}
By shrinking the set $U$ around $B$ if necessary, we may assume that it satisfies the definition of a special Cartan pair (Def. \ref{Cartan-pair} (III)). In particular, we have a holomorphic coordinate map $\phi\colon U\to U'\subset \C^n$ such that the sets $\phi(B)$, $\phi(A\cap B)$ and $\phi(U\cap (A\cup B))$ are strongly convex with $\cC^2$ boundaries in $U'$. (The last set may have corners in $\C^n$, but this is unimportant.) 

Choose a compact cube $Q \subset  W$ containing $0\in \C^N$ in its interior. The product $\phi(A\cap B)\times Q$ is then a compact convex set in $\C^{n+N}$. Since the manifold $Y$ enjoys CAP, we can approximate the map $F'$ with values in $Y$, uniformly on a neighborhood of $(A\cap B) \times Q$ in $\C^{n+N}$, by a holomorphic map $G' \colon U'\times Q\to Y$, where $U'$ is an open set in $X$ satisfying $B \subset U'\subset U$. 

It may be worth emphasizing that this is the unique place in the proof of Theorem \ref{th:CAP} where CAP of the fiber is invoked!

\smallskip {\em Step 3: Gluing.} Choose a number  $0<r<1$. If the approximation in Step 2 is sufficiently close then by Proposition \ref{gluing-sprays} we can glue $F$ and $G$ into a holomorphic spray of sections $\wt F \colon V' \times rQ \to Z$, where  $V'\subset X$ is an open neighborhood of $A\cup B$. The core section $\tilde f=\wt F(\cdotp,0) \colon V'\to Z$ then satisfies the conclusion of Lemma \ref{extending-to-bump}. 

Note that Proposition \ref{gluing-sprays} applies verbatim if $(A,B)$ is a strongly pseudoconvex Cartan pair in a Stein manifold $X$; this will be the case in our application of this lemma in the following subsection. The general case when $X$ is a complex space reduces to the special case by embedding a neighborhood of $A\cup B$ in $X$ as a subvariety $\Sigma$ in a Euclidean space $\C^n$ and approximating $(A,B)$ in the embedded picture from the outside by a strongly pseudoconvex Cartan pair in $\C^n$. The transition map between the two sprays, initially defined over a neighborhood of $A\cap B$ in $X$, extends holomorphically to a neighborhood of $A\cap B$ in $\C^n$, for example by using a holomorphic retraction onto a neighborhood of $A\cap B$ in $\Sigma$, or by using a bounded extension operator. (See the proof of Proposition 5.8.4 in \cite[p.\ 214]{FF:book}.) 
\end{proof}

%%%%%%%%%%%%%%%%%%%%%%%%%%%%%%%%%%%%%%%%%%%%%%%%%%%%%%%%%%%%%%%%%
%																																%
%   PROOF OF THE MAIN THEOREM    												        %
%																																%
%%%%%%%%%%%%%%%%%%%%%%%%%%%%%%%%%%%%%%%%%%%%%%%%%%%%%%%%%%%%%%%%%	
\subsection{Proof of the Oka principle}
\label{ss:proof}
We are now ready to prove the following basic case of Theorem \ref{th:CAP}. The more advanced versions (including interpolation on a subvariety, the parametric case, for sections of stratified fiber bundles, etc.) are obtained by similar analytic tools, but adapting the geometric construction to the case at hand. The proof given here and its generalizations can be found in \S 5.10--\S 5.13 of \cite{FF:book}.

\begin{theorem}[\bf The Basic Oka Property with Approximation]

\label{th:BOP}
Let $\pi\colon Z\to X$ be a holomorphic fiber bundle over a Stein manifold $X$ whose fiber $Y$ enjoys \CAP. Then every continuous section $f_0\colon X\to Z$ is homotopic to a holomorphic section $f_1\colon X\to Z$; if $f_0$ is holomorphic in a neighborhood of a compact $\cO(X)$-convex subset $K\subset X$ then the homotopy $f_t\colon X\to Z$ $(t\in [0,1])$ from $f_0$ to $f_1$ can be chosen so that $f_t$ is holomorphic near $K$ and uniformly close to $f_0$ on $K$ for each $t$.
\end{theorem}

\begin{proof}
Pick an open neighborhood $U\subset X$ of the set $K$ such that $f_0$ is holomorphic on $U$. Choose a smooth \spsh\ Morse exhaustion function $\rho\colon X\to\R$ such that $\rho<0$ on $K$ and $\rho>0$ on $X\setminus  U$. Let $p_1,p_2,p_3,\ldots$ be the critical points of $\rho$ in $\{\rho>0\}$, ordered so that $0<\rho(p_1)<\rho(p_2)<\rho(p_3)<\cdots$. Choose a sequence of numbers $0 = c_0 < c_1 < c_2<\cdots$ with $\lim_{j\to\infty} c_j=+\infty$ such that $c_{2j-1} < \rho(p_j) < c_{2j}$ for every $j=1,2,\ldots$, and these two values are sufficiently close to $\rho(p_j)$ (this condition will be specified later). If there are only finitely many $p_j$'s, we choose the remainder of the sequence $c_j$ arbitrarily. We subdivide the parameter interval $[0,1]$ of the homotopy into subintervals $I_j=[t_j,t_{j+1}]$ with $t_j = 1-2^{-j}$ $(j=0,1,2,\ldots)$. Choose a distance function $\dist$ on $Z$ induced by a complete Riemannian metric. Fix a number $\epsilon>0$. 

We shall inductively construct a homotopy of sections $f_t\colon X\to Z$ $(0\le t<1)$ such that for every $j\in\Z_+$ and $t\in [t_j,t_{j+1}]$ the section $f_t$ is holomorphic in a neighborhood of the set $K_j=\{x\in X\colon \rho(x)\le c_j\}$ and satisfies
\[
    \sup \bigl\{
    \dist(f_t(x),f_{t_j}(x)) \colon x\in K_j,\ t\in [t_j,t_{j+1}] \bigr\}
    < 2^{-j-1}\epsilon.
\]
The limit section $f_1=\lim_{t\to 1}f_t\colon X\to Z$ is then holomorphic on $X$ and satisfies 
\[
	\sup\{\dist(f_1(x),f_0(x))\colon x\in K_0\} <\epsilon. 
\]
Assuming inductively that a homotopy $\{f_t\}$ with the stated properties has been constructed for $t\in [0,t_j]$, we explain how to find it for $t\in [t_j,t_{j+1}]$.

\smallskip
{\em The noncritical case:}
If $j$ is even then $\rho$ has no critical points in the set 
\[
	K_{j+1}\setminus \mathring K_j=\{x\in X\colon c_j\le \rho(x)\le c_{j+1}\}. 
\]
By Lemma \ref{lem:extension-by-bumps} there are compact sets $K_j=A_0\subset A_1 \subset \cdots \subset A_m=K_{j+1}$ such that for every $k=0,1,\ldots,m-1$ we have $A_{k+1}=A_k\cup B_k$, where $B_k$ is a convex bump on $A_k$. In addition, we may choose the bumps $B_k$ small enough such that the bundle $Z\to X$ is trivial in an open neighborhood of any $B_k$. By Lemma \ref{extending-to-bump} we can successively extend the section across each bump $B_k$, approximating the previous section on $A_k$. In finitely many steps we thus obtain a holomorphic section $f_{t_{j+1}}$ in a neighborhood of $K_{j+1}$ which approximates $f_{t_j}$ as closely as desired on $K_j$. By the construction, the two sections are homotopic to each other over a neighborhood of $K_j$ so that the entire homotopy remains close to $f_{t_j}$. (In fact this is always true if these sections are close enough over a neighborhood of $K_j$.) Since there is no change of topology from $K_j$ to $K_{j+1}$, this homotopy can be extended continuously, first to a neighborhood of $K_{j+1}$, and then to all of $X$, by applying a cut-off function in the parameter of the homotopy.

\smallskip
{\em The critical case:} Now $j$ is odd, $j=2l-1$, and $\rho$ has a unique critical point $p=p_l$ in $\mathring K_{j+1}\setminus K_{j}$. 

We begin with some geometric considerations. The following well-known lemma gives the quadratic normal form for strongly plurisubharmonic functions at Morse critical points; see \cite[Lemma 3.9.1, p.\ 88]{FF:book} and the references therein.

\begin{lemma}
\label{spsh-normal-form}
Let $\rho$ be a strongly plurisubharmonic function of class $\cC^2$ in a neighborhood of the origin in $\C^n$, with a Morse critical point of index $k$ at $0\in\C^n$. Then $k\in\{0,1,\ldots,n\}$. Write $z=(z',z'') = (x'+\I y',x''+\I y'') \in \C^k\times \C^{n-k}$. After a $\C$-linear change of coordinates on $\C^n$ we have
\begin{equation}
\label{eqn:psh-normal}
        \rho(z)= \rho(0) -|x'|^2 + |x''|^2 +\sum_{i=1}^n \lambda_i y_i^2 +o(|z|^2)
\end{equation}
where $\lambda_i > 1$ for $i\in \{1,\ldots,k\}$ and $\lambda_i\ge 1$ for $i \in\{k+1,\ldots,n\}$.
\end{lemma}

Choose an open neighborhood $U\subset X$ of the point $p=p_l$ and a coordinate map $\phi \colon U\to P$ onto a polydisc $P\subset \C^n$ such that the function $\wt \rho=\rho\circ\phi^{-1} \colon P\to \R$ is given by (\ref{eqn:psh-normal}). By a small modification of $\rho$ supported in a neighborhood of $p$ we may also assume that the remainder term vanishes identically. 

If $k=0$ then $\rho$ has a local minimum at $p$. In this case the set $\{\rho \le c_{j+1}\}$ is the disjoint union of $\{\rho \le c_{j}\}$ and another connected component $W_j$ which appears around the point $p$. Assuming as we may that $c_{j}-c_{j-1}>0$ is small enough, the bundle $Z\to X$ is trivial over $W_j$, and we may extend the section $f_{t_j}$ in an arbitrary way to $W_j$ (for example, as a constant section in some local trivialization of the bundle). This gives a holomorphic section $f_{t_{j+1}}$ of $Z$ over $K_{j+1}$ and the induction may proceed. 

In the sequel we assume that $k\ge 1$. Pick a number $\alpha>0$ such that the set
\[
    X_{\alpha} = \{x\in X \colon \rho(p) -\alpha < \rho(x) < \rho(p) + 3\alpha \}
\]
does not contain any critical point of $\rho$ other than $p$, and we have 
\[
        \bigl\{ (x'+\I y',x''+\I y'')\in \C^k\times \C^{n-k} \colon
             |x'|^2 \le \alpha,\ |x''|^2 + \sum_{i=1}^n \lambda_i y_i^2 \le 4\alpha \bigr\}  \subset P.
\]
The stable manifold of the critical point at $0$ is the totally real subspace $\R^k\subset \C^n$ given by $y'=0,\ z''=0$. Let
\begin{equation}
\label{eqn:discE}
    E' = \{(x'+\I y',z'')\in\C^n \colon  y'=0,\ z''=0,\ |x'|^2\le \alpha \};
\end{equation}
this is a $k$-dimensional totally real disc whose boundary sphere $bE'$ is contained in the level set $\{\wt \rho=-\alpha\}$. The preimage $E=\phi^{-1}(E')\subset X$ is a local stable manifold at $p$ with $bE\subset \{\rho=\rho(p)-\alpha\}$. By the noncritical case we may assume that the number $c_j$ (with the property that a holomorphic section already exists on $\{\rho\le c_j\}$) is so close to $\rho(p)$ that $\rho(p) - \alpha < c_j$.

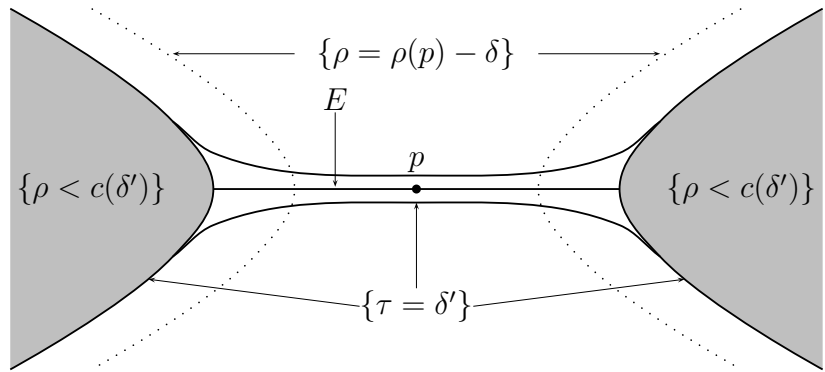
\begin{figure}[ht]
\psset{unit=0.6cm, xunit=1.8, linewidth=0.7pt} %%
\begin{pspicture}(-6.5,-4.3)(6,5)

%
%
% Geometry - colored
%
%

\pscustom[fillstyle=solid,fillcolor=lightgray]
{\pscurve[liftpen=1](5,4)(3,1.5)(2.5,0)(3,-1.5)(5,-4)                 % right hyperbola, yellow crosshatch
}

\pscustom[fillstyle=solid,fillcolor=lightgray]
{
\pscurve[liftpen=1](-5,4)(-3,1.5)(-2.5,0)(-3,-1.5)(-5,-4)             % left hyperbola colored
}

%
%  The core disc and the circle
%
%
\psline(-2.5,0)(2.5,0)                                            % the core disc
\psecurve(5,4)(3,1.5)(2.5,0.8)(0,0.3)(-2.5,0.8)(-3,1.5)(-5,4)         % upper boundary of L
\psecurve(5,-4)(3,-1.5)(2.5,-0.8)(0,-0.3)(-2.5,-0.8)(-3,-1.5)(-5,-4)  % lower boundary of L

\pscurve[linestyle=dotted](4,4)(2,1.5)(1.5,0)(2,-1.5)(4,-4)
% right hyperbola, dotted
\pscurve[linestyle=dotted](-4,4)(-2,1.5)(-1.5,0)(-2,-1.5)(-4,-4)
% left hyperbola, dotted

\rput(0,3){$\{\rho = \rho(p)-\delta \}$}
\psline[linewidth=0.2pt]{->}(1.5,3)(3,3)
\psline[linewidth=0.2pt]{->}(-1.5,3)(-3,3)

%
%
%   NOTATION
%
%
\rput(4,0) {$\{\rho < c(\delta')\}$}
\rput(-4,0){$\{\rho < c(\delta')\}$}

\psline[linewidth=0.2pt]{->}(0,-2.2)(0,-0.3)
\psline[linewidth=0.2pt]{<-}(-3.3,-2)(-0.7,-2.6)
\psline[linewidth=0.2pt]{->}(0.7,-2.6)(3.3,-2)
\rput(0,-2.6){$\{\tau = \delta'\}$}

\psline[linewidth=0.2pt]{->}(-1,1.7)(-1,0.05)
\rput(-1,2){$E$}

\psdot(0,0)
\rput(0,0.6){$p$}

\end{pspicture}
\caption{The level sets of $\tau$. (\cite[p.\ 94, Fig.\ 3.5]{FF:book})}
\label{Fig:criticallevel}
\end{figure}

For a sufficiently small $\delta >0$ there exists a smooth \spsh\ function $\tau\colon \{x\in X\colon \rho(x) < \rho(p) + 3\delta\} \to \R$ with the following properties:
\begin{itemize}
\item[\rm (i)]
$\{\rho\le c_j\} \cup E \subset \{\tau \le 0\} \subset \{\rho \le \rho(p) -\delta\} \cup E$,
\item[\rm (ii)]
$\{\rho \le \rho(p)+\delta\} \subset \{\tau \le 2\delta \} \subset \{\rho < \rho(p)+3\delta\}$, and
\item[\rm (iii)]
$\tau$ has no critical values in $(0,3\delta)$.
\end{itemize}
A typical level set $\{\tau=\delta'\}$ for small $\delta'>0$ is shown in Fig.\ \ref{Fig:criticallevel}. Outside a neighborhood of $p$ the level set $\{\tau=\delta'\}$ coincides with a certain level set $\{\rho=c(\delta')\}$ of $\rho$. 

We outline the construction of $\tau$ and refer for the details to \cite[\S 3.10]{FF:book}. 

Set $\lambda=\min\{\lambda_1,\ldots,\lambda_k\}>1$. Pick a number $\mu \in (1,\lambda)$ and set 
\[
	t_0 = t_0(\alpha,\mu) = \alpha \bigl(1- 1/\mu \bigr)^2 \in (0,\alpha).
\]
It is easy to find a smooth convex increasing function $h\colon \R \to [0,+\infty)$ satisfying the following conditions (see \cite[p.\ 92]{FF:book}): 
\begin{itemize}
\item[]
\begin{enumerate}
\item[(i)]   $h(t)=0$ for $t\le t_0$,
\item[(ii)]  for $t\ge \alpha$ we have $h(t)=t - t_1$ with $t_1=\alpha - h(\alpha) \in (t_0,\alpha)$,
\item[(iii)] for $t_0\le t\le \alpha$ we have $t-t_1 \le h(t) \le t - t_0$, and
\item[(iv)]  for all $t\in\R$ we have $0\le \dot h(t) \le 1$ and $2t\ddot h(t) + \dot h(t) < \lambda$.
\end{enumerate}
\end{itemize}
These properties of $h$ imply that the smooth function $\wt \tau \colon \C^n\to \R$ defined by
\[
    \wt \tau(z)= - h(|x'|^2) + |x''|^2 + \sum_{j=1}^n \lambda_j y_j^2
\]
is strongly plurisubharmonic (see \cite[p.\ 93]{FF:book}). Obviously $\wt \tau$ has no critical values in $(0,+\infty)$. Set $U'=\{q\in U \colon |x'(q)|^2 < \alpha\} \subset X$ and $V = \{\rho<\rho(p)+3\alpha\} \subset X$. We define a function $\tau \colon V \to\R$ as follows:
\[
    \tau(x) = \begin{cases} \rho(p)+\wt\tau(z(x)),     & x\in U\cap V; \cr
                     \rho(x) + t_1,                    & x\in V\setminus U.  
              \end{cases}       
\]
Here $t_1=\alpha - h(\alpha)$ is as in (ii) above. The properties of $h$ ensure that the two definitions agree on $(U\setminus U')\cap V$. It is easily verified that $\tau$ satisfies the stated conditions.

We now explain how to find the next holomorphic section $f_{t_{j+1}}$. We proceed in four steps, hence dividing the parameter interval $[t_j,t_{j+1}]$ into four subintervals. The number $\delta>0$ is as above (see the properties of $\tau$).

\smallskip
{\em Step 1:}  By the noncritical case we may assume that the section $f_{t_j}\colon X\to Z$ is holomorphic on a neighborhood of the set $\{\rho\le \rho(p)-\delta\} \subset X$.

\smallskip
{\em Step 2:} By Theorem 3.7.2 in \cite[p.\ 81]{FF:book} we can approximate $f_{t_j}$, uniformly on the set $\{\rho\le \rho(p) -\delta\}\cup E$, by sections that are holomorphic in small open neighborhoods of this set in $X$. The cited theorem is a local Mergelyan approximation theorem for manifold-valued maps on a union of a compact strongly pseudoconvex Stein domain $D\subset X$ and a smooth totally real submanifold $E\subset X\setminus \mathring D$ which is attached to the domain $J$-orthogonally along a Legendrian submanifold $bE\subset bD$.

\smallskip
{\em Step 3:} 
By property (i) of $\tau$ there is a $\delta'>0$ such that the set $\{\tau \le \delta'\}$ is contained in the neighborhood of $\{\rho\le \rho(p) -\delta\}\cup E$ on which the section from Step 2 is holomorphic. Applying the noncritical case with the function $\tau$ we deform the section from Step 2 to a section that is holomorphic on a neighborhood of  $\{\tau\le 2\delta\}$. 

\smallskip
{\em Step 4:} By property (ii), the section from Step 3 is holomorphic on $\{\rho\le \rho(p)+\delta\}$. Applying the noncritical case, this time again with the function $\rho$, we deform it to a holomorphic section on the set $K_{j+1}=\{\rho\le c_{j+1}\}$. 

\smallskip
These four steps together complete the construction of $f_{t_{j+1}}$, so the induction may proceed. This completes the proof of Theorem \ref{th:BOP}.
\end{proof}

\smallskip \noindent
{\bf Acknowledgements.}
My sincere thanks go to the organizers of the Winter School KAWA-4 in Toulouse in January 2013 for their kind invitation to give lectures and to prepare this survey. I wish to thank Finnur L\'arusson for writing the appendix to this paper and for his valuable suggestions on the presentation. Kind thanks also to Alexander Hanysz and Tyson Ritter for their careful reading of an initial draft, and to Peter Landweber for numerous suggestions that helped me to improve the presentation. Finally, I thank the referee for the remarks and references, and for suggesting that I include the new developments obtained since the time of the first submission.

%%%%%%%%%%%%%%%%%%%%%%%%%%%%%%%%%%%%%%%%%%%%%%%%%%%%%%%%%%%%%%%%%
%																																%
%   APPENDIX: A HOMOTOPY THEORETIC VIEWPOINT										%
%																																%
%%%%%%%%%%%%%%%%%%%%%%%%%%%%%%%%%%%%%%%%%%%%%%%%%%%%%%%%%%%%%%%%%	

\section{Appendix: The homotopy-theoretic viewpoint (by Finnur L\'arusson)}
\label{sec:Appendix}

\subsection{The connection with abstract homotopy theory}
\label{ss:connection}
Oka manifolds and Oka maps naturally fit into an abstract homotopy-theoretic framework, not merely by way of analogy, but in precise, rigorous terms.  This appendix, which is essentially Section 7 of the survey \cite{FLsurvey} expanded by a factor of two, provides an overview of how this comes about.  The papers in which the connection between abstract homotopy theory and Oka theory was developed are  \cite{Larusson1, Larusson2, Larusson3, Larusson4}.

Abstract homotopy theory, also known as homotopical algebra, was founded by D.~Quillen in his 1967 monograph \cite{Quillen1967}.  The fundamental notion of the theory is the concept of a model category, or a model structure on a category.  A model structure is an abstraction of the essential features of the category of topological spaces that make ordinary homotopy theory possible.  Model structures are good for many things.  They have been introduced and applied in various areas of mathematics outside of homotopy theory, for example in homological algebra, algebraic geometry, category theory, and theoretical computer science.  Here we view them as a tool for studying lifting and extension properties of holomorphic maps.  Model structures provide a framework for investigating two classes of maps such that the first has the right lifting property with respect to the second and the second has the left lifting property with respect to the first in the absence of topological obstructions.  It is more natural, in fact, to consider homotopy lifting properties rather than plain lifting properties, that is, liftings of families of maps varying continuously with respect to a parameter in a nice parameter space rather than liftings of individual maps.

One of the main results of Gromov in his seminal paper \cite[\S 2.9]{Gromov:OP} suggests a link with homotopical algebra.  Let $T\hra S$ be the inclusion into a Stein manifold $S$ of a closed complex submanifold $T$ (we call such an inclusion a \textit{Stein inclusion}), and let $X\to Y$ be a holomorphic fibre bundle whose fibre is an elliptic manifold (let us call such a map an \textit{elliptic bundle}).  Consider a commuting square
\[ \xymatrix{
T \ar[r] \ar[d] & X \ar[d] \\ S \ar[r] & Y
}\]
where $T\to X$ and $S\to Y$ are otherwise arbitrary holomorphic maps.  A basic version of Gromov's Oka principle states that every continuous lifting in the square, that is, every continuous map $S\to X$ such that the diagram
\[ \xymatrix{
T \ar[r] \ar[d] & X \ar[d] \\ S \ar[ur] \ar[r] & Y
}\]
commutes, can be deformed through such liftings to a holomorphic lifting.  Since $T\hra S$ is a topological cofibration and $X\to Y$ a topological fibration, by elementary homotopy theory there is a continuous lifting in the square, and hence by Gromov's theorem a holomorphic lifting, if one of the two vertical maps is a homotopy equivalence (let us call a homotopy equivalence an \textit{acyclic map}).  Thus, elliptic bundles have the right lifting property with respect to acyclic Stein inclusions, and acyclic elliptic bundles have the right lifting property with respect to Stein inclusions.

Compare this with one of Quillen's axioms for a model category (see \S \ref{ss:model-categories}):  A lifting $S \to X$ exists in every commuting square
\[ \xymatrix{
T \ar[r] \ar[d] & X \ar[d] \\ S \ar[r] & Y
}\]
in which $T\to S$ is a cofibration, $X\to Y$ is a fibration, and one of them is acyclic.  It is now natural to ask whether there is a model category containing the category of complex manifolds (it is too small to carry a model structure itself) in which Stein inclusions are cofibrations, elliptic bundles are fibrations, and weak equivalences are defined topologically.  

The answer is affirmative.  There is a natural, explicit way to embed the category of complex manifolds into a model category such that Gromov's theorem becomes an instance of Quillen's axiom.  In fact, a holomorphic map is a fibration in this model structure if and only if it is an Oka map \cite[Corollary 20]{Larusson2}.  In particular, a complex manifold is fibrant as an object in the model category (meaning that the map to the terminal object is a fibration) if and only if it is Oka.  Also, a complex manifold is cofibrant if and only if it is Stein \cite[Theorem 6]{Larusson3}.

\subsection{Model categories and simplicial sets}
\label{ss:model-categories}
A model category is a category with all small limits and colimits and three distinguished classes of maps, called weak equivalences or acyclic maps, fibrations, and cofibrations, such that the following axioms hold.
\begin{enumerate}
\item[(A1)]  If $f$ and $g$ are composable maps, and two of $f$, $g$, $f\circ g$ are acyclic, then so is the third.
\item[(A2)]  The classes of weak equivalences, fibrations, and cofibrations are closed under retraction.  (Also, it follows from the axioms that the composition of fibrations is a fibration, and the pullback of a fibration by an arbitrary map is a fibration.) 
\item[(A3)] A lifting $S\to X$ exists in every commuting square
\[\xymatrix{ T \ar[r] \ar[d] & X \ar[d] \\ S \ar[r] \ar@{-->}[ur] & Y }\]
in which $T\to S$ is a cofibration, $X\to Y$ is a fibration, and one of them is acyclic.
\item[(A4)]  Every map can be functorially factored as
\[\text{acyclic fibration}\,\circ\,\text{cofibration}\] 
and as
\[\text{fibration}\,\circ\,\text{acyclic cofibration.}\]
\end{enumerate}
For the theory of model categories, we refer the reader to \cite{Dwyer-Spalinski, Hirschhorn, Hovey, May-Ponto}.

There are many examples of model categories.  A fundamental example, closely related to the category of topological spaces, is the category of simplicial sets.  Simplicial sets are combinatorial objects that have a homotopy theory equivalent to that of topological spaces, but tend to be more useful or at least more convenient than topological spaces for various homotopy-theoretic purposes.  In homotopy-theoretic parlance, the distinction between topological spaces and simplical sets is blurred and the latter are often referred to as spaces.  The prototypical example of a simplicial set is the singular set $sX$ of a topological space $X$.  It consists of a sequence $sX_0, sX_1, sX_2, \ldots$ of sets, where $sX_n$ is the set of $n$-simplices in $X$, that is, the set of all continuous maps into $X$ from the standard $n$-simplex
\[T_n=\{(t_0,\ldots,t_n)\in\R^{n+1}:t_0+\cdots+t_n=1,\, t_0,\ldots,t_n\geq 0\},\]
along with face maps $sX_n\to sX_{n-1}$ and degeneracy maps $sX_n\to sX_{n+1}$.
The weak homotopy type of $X$ is encoded in $sX$.  For an introduction to simplicial sets, we refer the reader to \cite{Goerss-Jardine, May}.

\subsection{Complex manifolds as prestacks on the Stein site}
\label{ss:prestacks}
We have claimed that the embedding of the category of complex manifolds into a model category that realizes the Oka property as fibrancy is natural and explicit, but it is still quite technical.  We shall give a sketch here; the details may be found in \cite{Larusson2}.

First of all, how could we expect to be able to do homotopy theory with complex manifolds in a way that takes not only their topology but also their complex structure into account?  The answer lies in the following key observations, which vastly expand the scope of homotopical algebra.
\begin{itemize}
\item  Not only can we do homotopy theory with individual spaces, but also with diagrams or sheaves of them.
\item  Manifolds and varieties can be thought of as sheaves of spaces, so we can do homotopy theory with them too.  The general idea is known as the Yoneda lemma.
\end{itemize}
This line of thought has found a spectacular application in V.\ Voevodsky's homotopy theory of schemes and the resulting proof of the Milnor conjecture \cite{Voevodsky}.

The gist of the Yoneda lemma, sometimes called the most basic theorem in mathematics, is that an object is determined up to isomorphism by its relationships with other objects, that is, by the system of arrows into it from all other objects.  More precisely, there is a full embedding of each small category $\mathscr M$ into the category of presheaves of sets on itself, taking an object $X$ to the presheaf $\mathscr M(\cdot, X)$.

In our case, Stein manifolds play a special role as sources of maps, so we think of a complex manifold $X$ as defining a presheaf $\mathscr O(\cdot, X)$ on the full subcategory $\mathscr S$ of Stein manifolds of the category $\mathscr M$ of complex manifolds.  The presheaf consists of the set $\mathscr O(S, X)$ of holomorphic maps $S\to X$ for each Stein manifold $S$, along with the precomposition map $\mathscr O(S_2, X)\to\mathscr O(S_1, X)$ induced by each holomorphic map $S_1\to S_2$ between Stein manifolds.  Even though $\mathscr S$ is much smaller than $\mathscr M$, it may be shown that the presheaf $\mathscr O(\cdot, X)$ determines $X$, so we have an embedding, in fact a full embedding, of $\mathscr M$ into the category of presheaves of sets on $\mathscr S$.

Each set $\mathscr O(S, X)$ carries the compact-open topology.  A map between such sets defined by pre- or postcomposition by a holomorphic map is continuous.  We may therefore consider a complex manifold $X$ as a presheaf of topological spaces on $\mathscr S$.  This presheaf has the property that as a holomorphic map $S_1\to S_2$ between Stein manifolds is varied continuously in $\mathscr O(S_1,S_2)$, the induced precomposition map $\mathscr O(S_2, X)\to\mathscr O(S_1, X)$ varies continuously as well.  We would like to do homotopy theory with complex manifolds viewed as presheaves with this property.

Somewhat unexpectedly, as explained in \cite[\S 3]{Larusson2}, there are solid reasons, beyond mere convenience, to rephrase the above entirely in terms of simplicial sets.  For the technical terms that follow, we refer the reader to \cite{Larusson2} and the references cited there.  To summarize, we turn $\mathscr S$ into a simplicial site and obtain an embedding of $\mathscr M$ into the category $\mathfrak S$ of prestacks on $\mathscr S$.  The basic homotopy theory of prestacks on a simplicial site was developed by B.~To\"en and G.~Vezzosi for use in algebraic geometry \cite{Toen-Vezzosi}.  A new model structure on $\mathfrak S$, called the {\it intermediate structure} and based on ideas of J.~F.~Jardine, later published in \cite{Jardine}, was constructed in \cite{Larusson2}.  It is in this model structure that Gromov's Oka principle finds a natural home.

The main results of \cite{Larusson2} along with Theorem 6 of \cite{Larusson3} can be summarized as follows.

\begin{theorem}
The category of complex manifolds and holomorphic maps can be embedded into a model category such that:
\begin{itemize}
\item  a holomorphic map is acyclic when viewed as a map in the ambient model category if and only if it is a homotopy equivalence in the usual topological sense.
\item  a holomorphic map is a fibration if and only if it is an Oka map.  In particular, a complex manifold is fibrant if and only if it is Oka.
\item  a complex manifold is cofibrant if and only if it is Stein.
\item  a Stein inclusion is a cofibration.
\end{itemize}
\end{theorem}

A characterization of those holomorphic maps that are cofibrations is missing from this result.  It may be that Stein inclusions and biholomorphisms are the only ones.

Knowing that Oka maps are fibrations in a model structure helps us understand and predict their behaviour.  For example, it is immediate by abstract nonsense that the composition of Oka maps is Oka, that a retract of an Oka map is Oka, and that the pullback of an Oka map by an arbitrary holomorphic map is Oka (it is easily seen that the pullback exists in $\mathscr M$ and agrees with the pullback in $\mathfrak S$).  Also, in any model category, the source of a fibration with a fibrant target is fibrant.  It follows that the source of an Oka map with an Oka target is Oka.  On the other hand, the fact that the image of an Oka map with an Oka source is Oka is a somewhat surprising feature of Oka theory not predicted by abstract nonsense, the reason being that the Oka property can be detected using Stein inclusions of the special kind $T\hra \C^n$, where $T$ is contractible.

\subsection{Fibrant and cofibrant models}
\label{ss:fibrant-and-cofibrant-models}
It is a familiar process in mathematics to associate to an object a closely related but better behaved object, with a good map between the two.  For example, to a topological space we can associate a CW approximation, to a simplicial set a Kan complex, and to a module over a ring a projective resolution.  These are examples of fibrant and cofibrant models (also known as approximations or resolutions).  A fibrant model for an object $X$ in a model category is a fibrant object $Z$ with an acyclic cofibration (or sometimes only an acyclic map) $X\to Z$.  Factoring the map from $X$ to the terminal object as an acyclic cofibration followed by a fibration using axiom A4 above, we see that $X$ has a fibrant model, and by axioms A1 and A3, any two fibrant models for $X$ are weakly equivalent.  The dual notion is that of a cofibrant model.

Thinking about fibrant and cofibrant models for complex manifolds in the model
structure described above leads to interesting questions.  As far as we know, these
concepts had not been considered previously.  A cofibrant model for a complex
manifold $X$ that lives in the category of complex manifolds (and not merely in the
ambient model category) is a cofibrant complex manifold $S$, that is, a Stein
manifold, with an acyclic fibration $S\to X$, that is, a surjective Oka map with
contractible fibres.  We call $S$ a \textit{Stein model} for $X$.  Note that $X$ is
Oka if and only if $S$ is elliptic.  The definitions of good and very good complex
manifolds (see \S \ref{ss:good}) were motivated by the notion of a cofibrant model. 
A very good manifold has a Stein model.  It is an open question whether all complex
manifolds do.

An example of a fibrant model for a Stein manifold $S$ is an acyclic Stein inclusion $S\hookrightarrow X$ such that the Stein manifold $X$ is Oka, that is, elliptic.  It is an open question whether every Stein manifold has such a fibrant model.  We are asking for a variant of the Remmert-Bishop-Narasimhan embedding theorem with a proper holomorphic embedding that preserves the homotopy type of $S$ and whose target retains the key properties of affine space of being Stein and elliptic.  T.~Ritter has proved that every open Riemann surface acyclically embeds into an elliptic manifold, and that when the surface is an annulus, the target may be taken to be $\mathbb C\times\mathbb C^*$  \cite{Ritter1,Ritter2}.  The latter statement is a special case of Ritter's result, proved using the embedding techniques of E.~F.~Wold, that every continuous map from a circular domain in $\mathbb C$ to $\mathbb C\times\mathbb C^*$ is homotopic to an embedding.  More general results, for finitely connected planar domains, are proved in \cite{Larusson-Ritter}.

\subsection{Affine simplices in Oka manifolds}
\label{ss:affine-simplices}
Motivated by Gromov's comments in \cite[\S 3.5.G, \S 3.5.G']{Gromov:OP}, the affine singular set $eX$ of a complex manifold $X$ was defined in \cite{Larusson4} as the simplicial set whose $n$-simplices for each $n\geq 0$ are the holomorphic maps into $X$ from the affine $n$-simplex 
\[A_n=\{(t_0,\dots,t_n)\in\C^{n+1}:t_0+\dots+t_n=1\},\]
viewed as a complex manifold biholomorphic to $\C^n$, with the obvious face maps and degeneracy maps.  If $X$ is Brody hyperbolic, then $eX$ is discrete and carries no topological information about $X$.  On the other hand, when $X$ is Oka, $eX$ is \lq\lq large\rq\rq.

A holomorphic map $A_n\to X$ is determined by its restriction to $T_n\subset A_n$, so we have a monomorphism, that is, a cofibration $eX\hookrightarrow sX$ of simplical sets.  When $X$ is Oka, $eX$, which is of course much smaller than $sX$, carries the weak homotopy type of $X$.  More precisely, the cofibration $eX\hookrightarrow sX$ is the inclusion of a strong deformation retract (\cite{Larusson4}, Theorem 1).  Even for complex Lie groups, this result appears not to have been previously known.

%
%
%  THE BIBLIOGRAPHY
%
%

\end{document}